\documentclass[11pt]{amsart}
\usepackage{setspace}
\usepackage{url,amssymb,amsmath,amscd,graphicx,fancyhdr,paralist,mathtools,accents}
\usepackage{hyperref}
\usepackage{enumitem}
\usepackage{gensymb}
\usepackage{yfonts}
\usepackage{tikz-cd}
\usepackage[marginratio=1:1]{geometry}

\newcommand\sbullet[1][.5]{\mathbin{\vcenter{\hbox{\scalebox{#1}{$\bullet$}}}}}

\newtheorem{thm}{Theorem}[section]
\newtheorem{defn}[thm]{Definition}
\newtheorem{claim}[thm]{Claim}
\newtheorem{lem}[thm]{Lemma}
\newtheorem{prop}[thm]{Proposition}
\newtheorem{cor}[thm]{Corollary}
\newtheorem{question}[thm]{Question}

\theoremstyle{definition}
\newtheorem*{remark}{Remark}

\theoremstyle{definition}
\newtheorem{exam}[thm]{Example}

\newcommand{\ds}{\displaystyle}

\newcommand{\bd}{\partial}

\newcommand{\rank}{\operatorname{rank}}

\newcommand{\gr}{\operatorname{Gr}}

\newcommand{\Isom}{\operatorname{Isom}}

\newcommand{\RR}{{\mathbb R}}

\newcommand{\TT}{{\mathbb T}}

\newcommand{\ZZ}{{\mathbb Z}}
\newcommand{\NN}{{\mathbb N}}

\newcommand{\BB}{{\mathbb B}}

\newcommand{\vs}{\vspace{2mm}}

\newcommand{\too}{\longrightarrow}
\newcommand{\mapstoo}{\longmapsto}

\newcommand{\mc}{\mathcal}
\newcommand{\wt}{\widetilde}
\newcommand{\wh}{\widehat}

\newcommand{\im}{\operatorname{im}}
\newcommand{\length}{\operatorname{length}}
\newcommand{\id}{\operatorname{id}}
\newcommand{\codim}{\operatorname{codim}}
\newcommand{\Diff}{\operatorname{Diff}}
\newcommand{\Haus}{\operatorname{Haus}}
\newcommand{\Diffr}{\operatorname{Diff}^r}
\newcommand{\Diffone}{\operatorname{Diff}^1}
\newcommand{\vol}{\operatorname{vol}}
\newcommand{\norm}[1]{\left\lVert#1\right\rVert}

\newcommand{\Aone}{\operatorname{(A1)}}
\newcommand{\Atwo}{\operatorname{(A2)}}
\newcommand{\Atwostar}{\operatorname{(A2^*)}}

\newcommand{\supp}{\operatorname{supp}}

\newcommand{\nj}{\operatorname{NJ}}
\newcommand{\ucirc}[1]{\accentset{\circ}{#1}}

\title[An average intersection estimate for families of diffeomorphisms]{An average intersection estimate for families of diffeomorphisms}
\author[A.\ Kodat]{Axel Kodat}
\author[M.\ Shub]{Michael Shub}

\begin{document}

\begin{abstract}
We show that for any sufficiently rich compact family $\mc{H}$ of $C^1$ diffeomorphisms of a closed Riemannanian manifold $M$, the average geometric intersection number over $h \in \mathcal{H}$ between $h(V)$ and $W$, for $V, W$ any complementary dimensional submanifolds of $M$, is approximately (i.e. up to a uniform multiplicative error depending only on $\mathcal{H}$) the product of their volumes. We also give a construction showing that such families always exist. 
\end{abstract}

\maketitle

\section{Introduction}

Let $M$ be a closed Riemannian manifold of dimension $n$, and let $\Diffr(M)$ denote the space of $C^r$ diffeomorphisms of $M$ with the $C^r$ topology. For $0 \leq k \leq n$, let $\gr_k(TM)$ denote the $k$-Grassmannian bundle over $M$, i.e. the set of $k$-planes tangent to $M$. 

In this note we study conditions on compact $C^1$ families in $\Diffone(M)$ sufficient to guarantee a kind of ``uniform mixing" on average of the (smooth) incidence geometry of $M$. By a \emph{$C^r$ family of diffeomorphisms of $M$} we will mean a pair $(\mc{H}, \psi)$, where $\mc{H}$ is a smooth manifold, possibly with boundary, and $\psi : \mc{H} \to \Diffr(M)$ is such that the associated evaluation map $ev_\psi : \mc{H} \times M \to M$ is $C^r$. This implies in particular that the parametrizing map $\psi$ is continuous. When a family $(\mc{H}, \psi)$ is fixed we will usually suppress the map $\psi$, using the same letter $h$ to denote both an element in $\mathcal{H}$ and its image under $\psi$; by this mild abuse of notation we will often write, for instance, $ev_\psi(h,p) = \psi(h)(p) = h(p)$. 

We will also assume that any such $\mc{H}$ comes equipped with a Riemannian metric, and thus a Riemannian density which we denote by $d\mc{H}$. Given these assumptions, our main result is the following estimate: 

\begin{thm} \label{thm:main}

Suppose $(\mathcal{H}, \psi)$ is compact $C^1$ family of diffeomorphisms of $M$ such that
    \vspace{2mm}
\begin{enumerate}[label=$\operatorname{(A\arabic*)}$]
    \item $ev : \mathcal{H} \times M \to M \times M$ defined by $ev(h,p) = (p, h(p))$ is a submersion;
    
    \vspace{3mm}
    
    \item $\ucirc{ev}_k : \ucirc{\mathcal{H}} \times \gr_k(TM) \to \gr_k(TM) \times \gr_k(TM)$ defined by $ev_k(h, \sigma) = (\sigma, dh(\sigma))$ is surjective for all $0 \leq k \leq n$.

    \end{enumerate}

    \vspace{2mm}

\noindent Then there is a constant $C = C(\mathcal{H}, \psi) \geq 1$ such that
\begin{equation} \label{eqn:bounds}
\dfrac1C \vol(V)\vol(W) \ \ \leq \ \ \ds\int_{h \in \mathcal{H}} \#(h(V) \cap W) \ d\mathcal{H} \ \ \leq \ \ C \vol(V)\vol(W). \\
\end{equation}
for all complementary dimensional submanifolds $V, W \subset M$.
\end{thm}

Assumption (A1) says equivalently that for all $p \in M$, the map $ev_{\psi, p} : \mc{H} \to M$, $h \mapsto \psi(h)(p)$ is a submersion, or (again equivalently) that for all $(h, p) \in \mc{H} \times M$,
$$
\partial_h(ev_\psi)_{(h,p)} : T_h \mc{H} \to T_{h(p)} M
$$ is surjective. Assumption (A2) says that for all $0 \leq k \leq n$ and $\sigma_1, \sigma_2 \in \gr_k(TM)$, there is an $h \in \ucirc{\mc{H}}$ such that $dh(\sigma_1) = \sigma_2$; in other words, the interior of $\mc{H}$ ``acts transitively" on $\gr_k(TM)$ for every $k$.

Here $\#(h(V) \cap W) \in \NN_0 \cup \{ \infty \}$ simply means the number of points in $h(V) \cap W$, i.e. the geometric intersection number between $h(V)$ and $W$. Note that (\ref{eqn:bounds}) implicitly includes the claim that the function
$$
\phantom{.} \hspace{2cm}  \mathcal{H} \ \ni \ h \ \ \xmapsto{ \ \ \ \ \ \ \ } \ \ \#(h(V) \cap W) \ \in \ \NN_{\geq 0} \cup \{ \infty \}
$$
is measurable, and integrable if both $V$ and $W$ have finite volume; in the latter case it implies that $\#(h(V) \cap W) < \infty$ for a.e. $h \in \mc{H}$. 

By the volume of a submanifold $S \subset M$ we will always mean its intrinsic volume as a Riemannian manifold with metric inherited from $M$, i.e. $\vol(S) = \int_S dS$, where $dS$ denotes the density induced by the restricted metric. Alternatively, if $\dim(S) = k$ we have $\vol(S) = \Haus_M^k(S)$, the $k$-dimensional Hausdorff measure of $S$ in $M$. 

Theorem \ref{thm:main} is actually a special case of a more general result, which applies whenever $\dim(V) + \dim(W) \geq \dim(M)$ and correspondingly replaces the cardinality of the intersection with its appropriate-dimensional Hausdorff measure. We focus on the complementary dimensional case for simplicity of exposition (and greater proximity to our original motivations), but the proof of the general case works similarly. For the precise statement and some remarks on its proof, see \S \ref{sec:gen}.

By analogy with the (exact) kinematic formulas of integral geometry (see e.g. \cite{chern}, \cite{calegari}, \cite{santalo}), we call (\ref{eqn:bounds}) a \emph{kinematic inequality} for the family $(\mc{H}, \psi)$. Taken alone, Theorem \ref{thm:main} is a relatively straightforward adaptation of these classical formulas---which obtain when $\mathcal{H} = G$ is a compact Lie group acting isometrically on $M$ and transitively on every $\gr_k(TM)$---to a specific (a priori somewhat artificial and potentially rare) non-homogeneous context. Its possible utility derives from the following existence result, which we prove in \S \ref{sec:construction}: 
\begin{thm} \label{thm:existence1}
    Every closed manifold $M$ admits an $\mc{H}$ and $\psi : \mc{H} \to \Diff^1(M)$ satisfying the hypotheses of Theorem \ref{thm:main}.  
\end{thm}

\noindent We emphasize that the $\mc{H}$ constructed here has non-empty boundary. This is the main motivation for allowing $\bd \mc{H} \neq \varnothing$ in the definition of $C^r$ families, and in fact Theorem \ref{thm:existence1} fails in general if one requires $\bd \mc{H} = \varnothing$. For details see \S \ref{sec:remH}.

\subsection{Background and motivation} Formulas of the type considered here have a long history in integral geometry and a well-developed theory in the setting of homogeneous spaces of compact Lie groups. From this perspective the basic antecedent to our result is

\begin{thm}[Poincar\'e's formula for homogeneous spaces] \label{thm:poincare}
Let $M^n$ be a closed Riemannian manifold with $G$ a compact Lie group acting smoothly and transitively on $M$ by isometries. Suppose the action is also transitive on $\gr_k(TM)$, for every $1 \leq k \leq n-1$. Then for every $k, \ell \geq 0$ satisfying $k+\ell \geq n$, there is a constant $C_{k, \ell} > 0$ such that for any submanifolds $V, W \subset M$ with $\dim(V) = k$ and $\dim(W) = \ell$, we have 
\begin{equation} \label{eqn:poin}
\ds\int_{g \in G} \vol(gV \cap W)\ dG = C_{k, \ell} \vol(V) \vol(W),
\end{equation}
where $dG$ denotes the Haar measure on $G$. 
\end{thm}

In the form given here this result is due to Howard \cite{howard}. The basic idea behind the proof is to view the left-hand side as a double integral over what we will later call the \emph{intersection manifold} associated to $V$ and $W$, i.e. the space of all intersections between $g(V)$ and $W$ for $g \in G$. By changing the order of integration one then (essentially) equates this integral with the average volume of the spaces $G_{x,y} = \{  g \in G \ : \ gx = y  \}$ over all $(x, y) \in V \times W$. (In other words, one replaces the average volume of intersections with the average volume of the spaces of maps realizing each possible intersection.) Since all such spaces are isometric for any bi-invariant metric on $G$ inducing the Haar measure, the result follows. For details, see \cite[\S 3]{howard} or \cite[\S A.4]{condition}. Our adaptation of this approach below is particularly indebted to \cite[\S 13]{bcss} and \cite[\S 4]{shubsmale5}. For a classic presentation of the foundational material, see \cite{santalo}. 

Observe that any action $G \curvearrowright M$ as in Theorem \ref{thm:poincare} can be viewed as a compact $C^\infty$ family of diffeomorphisms of $M$, namely $(G, \psi)$ where $\psi : G \ni g \mapsto (x \mapsto gx) \in \Diff^\infty(M)$. This family clearly satisfies (A1) and (A2), and thus Theorem \ref{thm:main} yields (\ref{eqn:poin}) \emph{up to a uniform multiplicative error}. The fact that this error can be eliminated in the setting of Theorem \ref{thm:poincare} depends on (among other things) the assumption that the family $(G, \psi)$ consists of isometries. For an analogous result applicable to general manifolds, however, this is too much to ask: if $(\mc{H}, \psi)$ satisfies (A1) and $\im(\psi) \subset \Isom(M)$ for some choice of Riemannian metric on $M$, then the action of $\Isom(M)$ on $M$ is automatically transitive and thus $M$ is homogeneous, i.e. $M \cong G/H$ for $G$ a compact Lie group and $H < G$ closed. An inequality incorporating some bounded multiplicative error therefore seems to be the most one could expect from an extension of Poincar\'{e}'s formula to non-homogeneous manifolds.

Our primary motivation for formulating such an extension comes from smooth dynamics.\footnote{For a similar generalization from a different direction, see the work of Fu in \cite{fu}, and especially his definition of an ``admissible measured family" of diffeomorphisms in \S 6. We regret that we became aware of this paper too late to explore the connection further.} Given a $C^1$ diffeomorphism $f : M \to M$ and a pair of complementary dimensional submanifolds $V, W \subset M$, one would like to understand how the growth of $\vol(f^n(V))$ affects the distribution of (a large compact piece of) the isotopy class of $f^n(V)$ with respect to $W$. Our results here imply that we can fix a compact family of diffeomorphisms $\mc{H}$ of $M$ so that the average number of intersections between $h(f^n(V))$ and $W$ has the same order of growth (in the sense of e.g. \cite{correa_pujals_2023}) as $\vol(f^n(V))$. In future work we hope to show how this could be a useful tool in a new approach to the second author's \emph{$C^r$ entropy conjecture} for the open case(s) of $r < \infty$.

\pagebreak

\subsection{Assumptions} \label{sec:ass} Theorem \ref{thm:main} is a ``backwards" theorem in the sense that we started with the estimate (\ref{eqn:bounds}) and looked for reasonable (and in particular abundantly realized) assumptions that would make it true. Here we collect some comments on our eventual selections (A1) and (A2).

First, we emphasize that the statement of Theorem \ref{thm:main} is not quite optimal, in two main ways:
\begin{enumerate}[label=(\roman*)]
    \item The upper bound in (\ref{eqn:bounds}) actually follows from (A1) alone; the surjectivity assumption (A2) is needed only for the lower bound. Note also that (A1), in contrast to (A2), is a purely local property of $(\mc{H}, \psi)$. Since the property (A1) is also easily seen to be invariant under (either left- or right-) translation of the family by a diffeomorphism $g$, this implies that an upper bound as in (\ref{eqn:bounds}) can be obtained for finite-dimensional slices of $\Diffone(M)$ concentrated arbitrarily close to any given diffeomorphism $f$.

    \item In the other direction, it follows from the proof that (A2) is somewhat stronger than necessary for our result. What we really need is the following: 
    \vspace{1mm}
    \begin{itemize}[label=$\operatorname{(A2^*)}$]
    \setlength{\itemindent}{2em}
    \item \label{eqn:A2star} For any $\sigma_p \in \gr_k(T M)$ and $\sigma_q \in \gr_{n-k}(TM)$, there is an $h \in \ucirc{\mathcal{H}}$ such \\
    \phantom{.}  \ \ \ \ that $h(p) = q$ and $dh(\sigma_p) \pitchfork \sigma_q$.
    \end{itemize}
    \vspace{1mm}
    For further discussion of this point see the remark after Claim \ref{claim:eta1pos}. We have opted to use $\Atwo$ in the statement of Theorem \ref{thm:main} nonetheless, primarily to emphasize the analogy with the transitivity assumption in Theorem \ref{thm:poincare}, and also because it is no more difficult in practice to construct a family satisfying $\Atwo$ than one satisfying the weaker condition $\Atwostar$.
\end{enumerate}

These qualifications in mind, the purpose of our assumptions is heuristically the following. The assumption that $ev$ is a submersion implies that no particular intersection $h(p) = q$ between $h(V) \cap W$ can be locally persistent (since one can move $h(p)$ off itself by a small change in $h$). In particular, any non-transversal intersection between $h(V)$ and $W$ (near which the cardinality of $h(V) \cap W$ could explode) can be destroyed by a small perturbation of $h$. This ensures that almost all intersections are isolated, and moreover that the total number of such points (summed over $\mc{H}$) is never too large compared to $\vol(V) \cdot \vol(W)$. This yields the upper bound in Theorem \ref{thm:main} using (A1) alone.

To get a lower bound, one needs to show that $\mc{H}$ can always \emph{produce} intersections between pairs of points in $V$ and $W$. Moreover, since the quantity we are trying to bound is a sort of average of the number of intersections with respect to the volume measure on $\mc{H}$, these intersections should persist over open (thus positive measure) subsets of $\mc{H}$. In other words, $\mc{H}$ should be able to produce \emph{transverse} intersections between any pair of points in $V$ and $W$. This will certainly hold if $\mc{H}$ can move any tangential $k$-plane to any other, as guaranteed by (A2). 

The following example cleanly illustrates the main ideas in the isometric case: 

\begin{exam}[$\TT^2 \curvearrowright \TT^2$]
Let $M = \TT^2 = \RR^2/\ZZ^2$ be the standard flat $2-$torus, and $\mathcal{H} = \TT^2$ with $\psi : \TT^2 \longrightarrow \Diff(\TT^2)$ the smooth action of $\TT^2$ on itself by translations; explicitly, $\psi : a \mapsto (\psi_a : x \mapsto x+a)$. Then the pair $(\mathcal{H}, \psi)$ satisfies assumption (A1), but not (A2), since $dh = \text{id}$ for all $h \in \text{im}(\psi)$ (using the canonical identifications $T_p \TT^2 \cong \RR^2$) and thus 
$$
\text{im}(ev_1) = \{ (\sigma_1, \sigma_2) \ : \ \sigma_1 \text{ and } \sigma_2 \text{ are parallel} \} \neq \gr_1(T\TT^2) \times \gr_1(T\TT^2).
$$ 
More strongly, the fact that every element of $\mathcal{H}$ maps any element of $\gr_1(T\TT^2)$ to a parallel translate means that no element of $\gr_1(T\TT^2)$ can be made transverse to itself via $\mathcal{H}$. 

Theorem \ref{thm:main} correspondingly fails. For instance, take $I$, $J$ two parallel Euclidean line segments in $\TT^2$. Then it is easy to see that $\psi_a(I) \cap J = \varnothing$ for a.e. $a \in \TT^2$, so
$$
\ds\int_{a \in \TT^2} \#(\psi_a(I) \cap J) \ d\TT^2 = 0.
$$
This implies that the lower bound in Theorem \ref{thm:main} cannot hold for any positive $C$. 

If the angle between $I$ and $J$ is $\theta \neq 0$, however, it is similarly straightforward to show that 
$$
\ds\int_{a \in \TT^2} \#(\psi_a(I) \cap J) \ d\TT^2 = |\sin \theta| \length(I) \cdot \length(J). 
$$
In particular, we have 
$$
\ds\int_{a \in \TT^2} \#(\psi_a(I) \cap J) \ d\TT^2 \leq \length(I) \cdot \length(J) 
$$
for any Euclidean line segments $I, J \subset \TT^2$, and by a subdivision and approximation argument it is not difficult to see that this holds more generally for $I, J \subset \TT^2$ any embedded $C^1$ curves. Thus the upper bound in Theorem \ref{thm:main} holds here with $C=1$. 
\end{exam}

\subsection{A digression on $\mc{H}$} \label{sec:remH} One technical hurdle in proving Theorem \ref{thm:main} stems from our allowance of the possibility that the parameter space $\mathcal{H}$ has boundary. If one assumes instead that $\bd \mathcal{H} = \varnothing$, the availability of Ehresmann's fibration theorem (applied to the submersion $ev$) substantially streamlines the proof; see the remark concluding \S \ref{sec:bundles}. However, here Ehresmann's theorem cuts both ways, imposing strong topological constraints on any boundaryless $\mathcal{H}$ satisfying $\Aone$. These constraints imply in particular that for general $M$, the condition $\bd \mathcal{H} = \varnothing$ is too much to ask:

\begin{prop} \label{prop:Hnoboundary}
Suppose $M$ is a closed manifold and $(\mc{H}, \psi)$ a compact, connected family of $C^1$ diffeomorphims of $M$ with $\bd \mc{H} = \varnothing$, such that $ev_p$ is a submersion for some $p \in M$. Then $\pi_1(M)$ is a finite extension of a quotient of a subgroup of $\pi_1(\Diff_0(M))$.
\end{prop} 

\begin{proof}
Let $(\mc{H},\psi)$ be such a family. After translating and perturbing we may assume that $\psi(\mc{H}) \subset \Diff_0(M)$ the identity component of $\Diff^\infty(M)$. Choose $p \in M$ such that $ev_p$ a submersion. 

By Ehresmann's theorem, $ev_p : \mc{H} \to M$ is a locally trivial fibration, say with fiber $\mc{F}$. The end of the homotopy long exact sequence for this fibration reads
$$
\pi_1(\mc{H}) \xrightarrow[]{ \ (ev_p)_* \ } \pi_1(M) \xrightarrow[ \ \ \ \ ]{} \pi_0(\mc{F}) \xrightarrow[ \ \ \ \ ]{} 1.
$$ 
Since $\mc{H}$ and (therefore) $\mc{F}$ are compact we must have $\# \pi_0(\mc{F}) < \infty$, so this sequence implies that $(ev_p)_*$ is virtually onto. On the other hand, $ev_p = EV_p \circ \psi$ where $EV_p : \Diff_0(M) \to M$, $f \mapsto f(p)$. Thus $(ev_p)_*$ factors through $\pi_1(\Diff_0(M))$, and the result follows. 
\end{proof}

\begin{cor} \label{cor:hypno}
Suppose $M$ is a closed hyperbolic manifold with $\dim(M) \leq 3$. Then there is no boundaryless compact family of diffeomorphisms of $M$ satisfying $\Aone$. 
\end{cor} 

\begin{proof} 
By work of Earle-Eells (in dimension 2) and Gabai (in dimension 3), $\Diff_0(M)$ is contractible (\cite{e-eells}, \cite{gabai}). In particular, $\pi_1(\Diff_0(M)) = 0$. Since (by hyperbolicity of $M$) $\pi_1(M)$ is infinite, by Proposition \ref{prop:Hnoboundary} we are done.
\end{proof}

The upshot here is that Theorem \ref{thm:existence1} fails if one requires that $\bd \mc{H} = \varnothing$. More strongly, these results (combined with the restrictions on higher homotopy groups obtained by continuing the above exact sequence) suggest that manifolds which admit a submersive compact family of diffeomorphisms without boundary may be quite rare. Clearly homogeneous spaces of compact Lie groups form the main class of examples. Are there others? 

\begin{question} \label{question:Hboundary}
Which closed manifolds $M$ admit a compact smooth family of diffeomorphisms $\mathcal{H} \longrightarrow \Diff(M)$ with $\bd \mathcal{H} = \varnothing$, such that the associated evaluation map $ev : \mathcal{H} \times M \longrightarrow M \times M$ is a submersion?
\end{question} 

 We note that, on the question of existence, the choice of regularity is irrelevant. A $C^r$ family of diffeomorphisms can be viewed as a map $\Phi \in C^r(\mc{H} \times M, M)$ such that $\Phi(h, \cdot) \in \Diffr(M)$ for all $h \in \mc{H}$; with this notation, the submersion condition simply means that $\Phi(\cdot, p) : \mc{H} \to M$ is a submersion for all $p \in M$. When $\mc{H}$ and $M$ are compact, these properties are open in $C^1(\mc{H} \times M, M)$, so if one obtains a boundaryless $C^1$ family satisfying $\Aone$, by perturbing one can obtain a $C^\infty$ family as well. 

This intriguing diversion aside, the proof of Theorem \ref{thm:main} in the general, possibly-with-boundary setting turns out to work in much the same way as the Ehresmann-based proof for the boundaryless case, modulo some basic lemmas concerning submersions from compact manifolds with boundary which we relegate to Appendix A.\footnote{Though these straightforward results are presumably well-known, we were unable to locate them in the literature.} The content of these lemmas is to show that our required bounds (which follow trivially from compactness when the fibers of $ev$ move $C^1$-continuously) still follow from our assumptions when we allow boundary, provided that the surjectivity requirement in (A2) is phrased in terms of the interior of $\mathcal{H}$. 

\subsection{Acknowledgements} We would like to thank Bryce Gollobit, Charles Pugh, Enrique Pujals, and Federico Rodriguez Hertz for many stimulating conversations and helpful suggestions on the contents of this paper. The first author would also like to thank Ryan Utke for pointing out the relevance of Ehresmann's theorem to these results. 

\pagebreak

\section{Preliminaries}

Let $(\mc{H},\psi)$ satisfy the hypotheses of Theorem \ref{thm:main}, and write $d := \dim(\mc{H}).$ We begin by introducing the main objects used in our proof.

\subsection{Normal Jacobians} \label{sec:normjac} 

Let $E_1, E_2$ be finite-dimensional real inner product spaces. Then for any linear map $A : E_1 \to E_2$, one defines the \emph{normal Jacobian} by the formula
$$
\nj(A) := \det(AA^*)^{1/2} \in \RR_{\geq 0}.
$$
Note that $\nj(A) > 0$ iff $A$ is surjective. One can check that
$$
\nj(A) = |\det ( A|_{\ker(A)^\perp} )|.
$$

For a $C^1$ map $f : M \to N$ between Riemannian manifolds, we define the \emph{normal Jacobian of $f$} pointwise, by
$$
\nj(f)(x) := \nj(df_x) = \det(df_x \cdot df_x^*)^{1/2}.
$$ 
Thus $\nj(f) : M \to \RR_{\geq 0}$, and since $f$ is $C^1$ it is easy to see that $NJ(f)$ is continuous. By the above, $\nj(f)(x) > 0$ iff $x$ is a regular point of $f$.

\subsection{The solution manifold $\widetilde{\mathcal{V}}$} \label{sec:solmfld} 

The \emph{solution manifold} $\widetilde{\mathcal{V}} \subset \mathcal{H} \times M \times M$ is the subspace consisting of all triples of the form $(h, p, h(p))$. Equivalently, $\widetilde{\mathcal{V}} = \Gamma(ev_\psi)$, where
\begin{align*}
ev_\psi \ : \ &\mathcal{H} \times M \longrightarrow M \\
&(h, p) \ \longmapsto \ h(p).
\end{align*}
Since $ev_\psi$ is $C^1$, $\widetilde{\mathcal{V}}$ is a neat $C^{1}$ submanifold with boundary of $\mathcal{H} \times M \times M$, diffeomorphic to $\mathcal{H} \times M$.\footnote{By \emph{neat} we mean that $\bd\widetilde{\mathcal{V}} \subset \bd (\mathcal{H} \times M \times M) = \bd \mathcal{H} \times M \times M$.} In particular, it inherits $C^{1}$ projection maps $\widetilde{\pi}_1, \widetilde{\pi}_2$: 
\begin{center}
\begin{tikzcd}
& \arrow{ld}[above]{\tilde{\pi}_1 \ \ \ } \widetilde{\mathcal{V}} \arrow[rd, "\tilde{\pi}_2"] & \\
\mathcal{H} & & M \times M.
\end{tikzcd}
\end{center}

\noindent By (A1) and the $k=0$ case of (A2), ${\tilde{\pi}_2}|_{\text{int}(\widetilde{\mathcal{V}})} : \text{int}(\widetilde{\mathcal{V}}) \to M \times M$ is a surjective submersion. For $(p, q) \in M \times M$ we write
$$
\mathcal{V}_{p,q} := \tilde{\pi}_2^{-1}(p,q) \cap \text{int}(\widetilde{\mathcal{V}}) = \left({\tilde{\pi}_2}|_{\text{int}(\widetilde{\mathcal{V}})} \right)^{-1}(p,q).
$$ 
Then each $\mathcal{V}_{p,q}$ is an embedded $(d-n)$-dimensional submanifold of $\text{int}(\widetilde{\mathcal{V}})$ with Riemannian metric inherited from $\mathcal{H} \times M \times M$ and induced Riemannian density which we will denote by $d\mathcal{V}_{p, q}$.

\subsection{The intersection manifolds $\mc{V}$} \label{sec:solmfld} 

Now let $V, W \subset M$ be complementary dimensional submanifolds of $M$, say $\dim(V) = k$ and $\dim(W) = n - k$. We define the \emph{intersection manifold} $\mc{V}$ associated to $\mc{H}$, $V$, and $W$ by 
\begin{align*}
\mathcal{V} &:= \{ (h, p, q) \in \ucirc{\mathcal{H}} \times V \times W \ : \ h(p) = q \} \\
&= \left( {\tilde{\pi}_2}\big|_{\text{int}(\widetilde{\mathcal{V}})} \right)^{-1}(V \times W) = \tilde{\pi}_2^{-1}(V \times W) \cap \text{int}(\widetilde{\mathcal{V}}).
\end{align*}
Since ${\tilde{\pi}_2}\big|_{\text{int}(\widetilde{\mathcal{V}})} : \text{int}(\widetilde{\mathcal{V}}) \to M \times M$ is a surjective submersion, $\mc{V}$ is a $C^1$ submanifold of $\text{int}({\widetilde{V}})$ with $\codim_{\wt{\mc{V}}}(\mc{V}) = \codim_{M \times M}(V \times W) = n$, so
$$
\dim(\mc{V}) = \dim(\wt{\mc{V}}) - n = \dim(\mc{H}) + n - n = \dim(\mc{H}).
$$
Again note that $\mc{V}$ inherits a Riemannian metric from $\mc{H} \times M \times M$; as usual we will write $d\mc{V}$ for the induced Riemannian density. 

For $i=1, 2$, define $\pi_i := {\tilde{\pi}_i}|_{\mathcal{V}}$. Then we have the diagram
\begin{center}
\begin{tikzcd}
& \arrow{ld}[above]{\pi_1 \ \ \ } \mathcal{V} \arrow[rd, "\pi_2"] & \\
\mathcal{H} & & V \times W
\end{tikzcd}
\end{center}
of $C^1$ maps, with $\pi_2$ again a surjective submersion. 

In the first step of the proof we will apply the co-area formula to this diagram to convert our main integral
$$
\ds\int_{h \in \mc{H}} \#(h(V) \cap W) \ d\mc{H}
$$
to an integral over $V \times W$ involving the normal Jacobians of $\pi_1$ and $\pi_2$. To estimate this latter integral, we will need a description of the tangent spaces
$$
T_{(h,p,q)} \mc{V} \subset T_h \mc{H} \times T_p V \times T_q W. 
$$
This is given easily by
\begin{claim} \label{claim:tangentspace} For all $(h,p,q) \in \mc{V}$ we have
$$
T_{(h,p,q)} \mathcal{V} = \{  (\dot{h}, \dot{p}, \dot{q}) \in T_h \mathcal{H} \times T_p V \times T_q W \ : \ J_{(h, T_p V, T_q W)}(\dot{p}, \dot{q}) = d(ev_p)_h(\dot{h}) \}
$$
where $J_{(h, T_p V, T_q W)} : T_p V \times T_q W \to T_q M$ is defined by
$$
J_{(h, T_p V, T_q W)}(v, w) = w - dh(v).
$$
\end{claim}

\begin{proof}
Differentiating the defining equation $h(p) = q$ yields
\begin{align*}
\dot{q} = d(ev_p)_h(\dot{h}) + dh(\dot{p}).
\end{align*}
The claim follows.
\end{proof}

Notice that, since $d(ev_p)_h : T_p \mc{H} \to T_{h(p)} M$ is surjective for every $(h, p) \in \mc{H} \times M$, the projection map
\begin{align*}
    d(\pi_1)_{(h,p,q)} \ : \ T_{(h,p,q)} \mc{V} &\too T_h \mc{H} \\
    (\dot{h}, \dot{p}, \dot{q}) &\mapstoo \dot{h}
\end{align*}
will be surjective if and only if $J_{(h,T_p,T_q W)}$ is surjective, i.e. iff $dh(T_p V) + T_q W = T_q M$. We summarize this observation in the following: 

\begin{prop} \label{prop:pi1critpts}
The regular points of the projection map $\pi_1 : \mc{V} \to \mc{H}$ are precisely those $(h,p,q) \in \mc{V}$ such that $h(V)$ intersects $W$ transversally at $q$. Thus $h \in \ucirc{\mc{H}}$ is a regular value of $\pi_1$ iff $h(V) \pitchfork W$.
\end{prop}

The restriction to the interior of $\mc{H}$ in the above statement is necessary because by definition $\mc{V} \subset \ucirc{\mc{H}} \times V \times W$, and thus every $h \in \bd \mc{H}$ is trivially a regular value of $\pi_1$, even though $\mc{V}$ gives no (immediate) information about the intersections $h(V) \cap W$ for such $h$. However, since $\bd \mc{H}$ has measure zero in $\mc{H}$ and $\dim(\mc{H}) = \dim(\mc{V})$, by applying Sard's theorem to the $C^1$ map $\pi_1 : \mc{V} \to \ucirc{\mc{H}}$ we get

\begin{cor} \label{cor:trans}
$h(V) \pitchfork W$ for almost every $h \in \mc{H}$.
\end{cor}

\begin{remark}
The argument given here essentially follows the usual proof of the classical transversality theorem. In low regularity, however, its appeal to Sard's theorem requires constraints on the dimensions of $V$ and $W$. In particular, in the $C^1$ category Corollary \ref{cor:trans} fails when $\dim(V) + \dim(W) > \dim(M)$. This difficulty is the main reason for considering Hausdorff measures instead of volumes of intersections in the generalization to Theorem \ref{thm:main}; see \S \ref{sec:gen} for details. 

\end{remark}

An immediate consequence of the formula above is that the tangent space $T_{(h,p,q)}\mc{V}$ depends only on $p, T_p V$ and $T_q W$. The same is therefore true of the linear maps $(d\pi_i)_{(h,p,q)}$ (themselves just restrictions of projections), and consequently their normal Jacobians. 

We would like to show that more is true, namely that $\nj(\pi_1)$ and $\nj(\pi_2)$ are continuous when viewed as functions of $h$, $T_p V$, and $T_q W$. The next section introduces the objects necessary to make this statement precise.

\subsection{The bundles $\mathcal{G}_k \to \widetilde{\mathcal{V}}$} \label{sec:bundles}

For each $0 \leq k \leq n$ let
$$
\gr_k(TM) = \ds\bigsqcup_{p \in M} \gr_k(T_p M)
$$
denote the $k$-Grassmanian bundle over $M$. This is a smooth fiber bundle over $M$, associated to the tangent bundle $TM$ in the sense that it can be obtained by gluing local trivializations $U_\alpha \times \gr_k(\RR^n)$ according to the smooth clutching functions $\rho_{\alpha, \beta} : U_\alpha \cap U_\beta \to GL(n, \RR)$ for $TM$. For each $k$ we let $p_k : \gr_k(TM) \to M$ denote the standard (smooth) projection map. For $\sigma \in \gr_k(TM)$ we will typically write $\sigma = \sigma_x$ for $x \in M$ to mean that $x = p_k(\sigma)$, i.e. $\sigma \in \gr_k(T_x M)$. In words, $\sigma_x \in \gr_k(TM)$ simply denotes a $k$-plane tangent to $M$ at the point $x$. 

Now define
$$
\mathcal{G}_k := \left\{ (h, \sigma_p, \sigma_q) \in \mathcal{H} \times \gr_k(TM) \times \gr_{n-k}(TM) \ : \ h(p) = q \right\}.
$$
Clearly $\mathcal{G}_k$ is a $C^1$ submanifold of $\mathcal{H} \times \gr_k(TM) \times \gr_{n-k}(TM)$, and indeed the map $\text{id} \times p_k \times p_{n-k}$ exhibits it as a $C^1$ $(\gr_k(\RR^n) \times \gr_{n-k}(\RR^n))$-bundle over $\widetilde{\mathcal{V}}$. (In particular, $\mathcal{G}_0 \cong \mathcal{G}_n \cong \widetilde{\mathcal{V}}$.)

\begin{claim} \label{claim:grasssub}
The map $\wh{\pi}^k_2 : \mathcal{G}_k \to \gr_k(TM) \times \gr_{n-k}(TM)$, $(h, \sigma_1, \sigma_2) \mapsto (\sigma_1, \sigma_2)$ is a $C^1$ submersion and surjective when restricted to $\ucirc{\mathcal{G}}_k$. 
\end{claim}

\begin{proof} This follows from the observation that $\mathcal{G}_k$ can be obtained as the pullback of the smooth fiber bundle $\gr_k(TM) \times \gr_{n-k}(TM) \to M \times M$  via the $C^1$ submersion $\tilde{\pi}_2 : \widetilde{\mathcal{V}} \to M \times M$. More precisely, we have the following diagram: 
\begin{center}
\begin{tikzcd}
 \mathcal{G}_k \arrow{dd}[left]{\text{id} \times p_k \times p_{n-k}} \arrow{rr}[above]{\wh{\pi}^k_2} & & \arrow{dd}[right]{p_k \times p_{n-k}} \gr_k(TM) \times \gr_{n-k}(TM) \\
  & & \\
\widetilde{\mathcal{V}} \arrow{rr}[below]{\tilde{\pi}_2} & & M \times M.
\end{tikzcd}
\end{center}
In local coordinates (i.e. above any chart $U \subset \widetilde{\mathcal{V}}$), $\hat{\pi}_k$ has the form 
$$
\hat{\pi}_k \ \cong \ \tilde{\pi}_2 \times \text{id} \ : \ U \times \left( \gr_k(\RR^n) \times \gr_{n-k}(\RR^n) \right) \ \longrightarrow \ \tilde{\pi}_2(U) \times \left( \gr_k(\RR^n) \times \gr_{n-k}(\RR^n) \right),
$$
which is clearly a submersion because $\tilde{\pi}_2$ is. The surjectivity of $\left. \hat{\pi}_k \right|_{\ucirc{\mathcal{G}}_k}$ follows immediately from that of $ev\big|_{\ucirc{\mc{H}} \times M}$, i.e. the $k=0$ case of $\Atwo$. \end{proof}

Now for each $(h, \sigma_p, \sigma_q) \in \mathcal{G}_k$ we can define
\begin{align*}
J_{(h, \sigma_p, \sigma_q)} : \sigma_p \times \sigma_q & \to T_q M \\
(v, w) & \mapsto w - dh(v),
\end{align*}
and
$$
E_k(h,\sigma_p, \sigma_q) := \{  (\dot{h}, v, w) \in T_h \mc{H} \times \sigma_p \times \sigma_q \ : \ J_{(h, \sigma_p, \sigma_q)}(v,w) = d(ev_p)_h(\dot{h}) \}. 
$$
Finally, set
\begin{align*}
    \Pi_1(h,\sigma_p, \sigma_q) \ : E_k(h,\sigma_p,\sigma_q) &\too T_h \mc{H} \\
    (\dot{h}, v, w) &\mapstoo \dot{h}, \\
    \Pi_2(h,\sigma_p,\sigma_q) \ : E_k(h,\sigma_p,\sigma_q) &\too \sigma_p \times \sigma_q \\
    (\dot{h}, v, w) &\mapstoo (v, w),
\end{align*}
and for $i=1,2$, define
$$
\eta_i : \mc{G}_k \too \RR_{\geq 0}
$$
by 
$$
\eta_i(h,\sigma_p,\sigma_q) = \nj(\Pi_i(h,\sigma_p,\sigma_q)). \vspace{2mm}
$$

\noindent The point of these objects, immediate from the definitions and Claim \ref{claim:tangentspace}, is the following: 

\begin{lem} \label{lem:etatonj}
    Let $V, W \subset M$ be complementary dimensional submanifolds with $\dim(V) = k$, and $\mc{V}$ their associated intersection manifold. Then for all $(h,p,q) \in \mc{V}$ and $i = 1, 2$ we have
    \begin{align*}
    T_{(h,p,q)} \mc{V} &= E_k(h,T_p V, T_q W); \\[2mm]
    d(\pi_i)_{(h, p, q)} &= \Pi_i(h, T_p V, T_q W); \\[2mm]
    \nj(\pi_i)_{(h,p,q)} &= \eta_i(h, T_p V, T_q W). 
    \end{align*}
    
\end{lem}
\vs
\noindent We will also need the following properties: 

\begin{claim}
    $E_k$ is a continuous vector bundle over $\mc{G}_k$ of rank $d$.
\end{claim}

\begin{proof} Let $\Omega_j$ denote the tautological vector bundle over $\gr_j(TM)$, i.e. the smooth rank-$j$ vector bundle $\Omega_j \to \gr_j(TM)$ with fibers $\Omega_j(\sigma) = \sigma$. Observe that, since each $\sigma \in \gr_j(TM)$ is a subspace of some tangent space of $M$, $\Omega_j$ carries a natural bundle metric obtained by restricting the Riemannian metric on $M$. 

Let $\mathcal{E}_k$ denote the restriction of the vector bundle $T\mathcal{H} \times \Omega_k \times \Omega_{n-k} \to \mc{H} \times \gr_k(TM) \times \gr_{n-k}(TM)$ to $\mc{G}_k$. Again this carries a natural bundle metric, namely the product metric induced from its factors. 

We can now view $J$ as a bundle map
\begin{align*}
J : \mathcal{E}_k \too \gamma^* TM,
\end{align*}
where $\gamma : \mc{G}_k \to M$ is given by $\gamma(h,\sigma_p,\sigma_q) = q$. Continuity of $J$ follows from that of our parametrization $\psi : \mc{H} \to \Diffone(M)$. Define
$$
H : \mathcal{E}_k \too \gamma^* TM
$$
by
\begin{align*}
 H_{(h, \sigma_p, \sigma_q)} : T_h \mc{H} \times \sigma_p \times \sigma_q &\too T_q M \\
(\dot{h}, v, w) &\mapstoo d(ev_p)_h(\dot{h}), 
\end{align*}
and note that $H$ is a continuous bundle map because $ev$ is $C^1$. Finally set $L := J - H$. 

Since $d(ev_p)_h$ is surjective for all $(h,p) \in \mc{H} \times M$ and $J_{(h,\sigma_p,\sigma_q)}$ does not depend on $\dot{h}$, $L_{(h, \sigma_p, \sigma_q)}$ is surjective for all $(h, \sigma_p, \sigma_q) \in \mc{G}_k$. In particular, $L$ is a constant rank vector bundle map, so $\ker(L)$ is a continuous vector subbundle of $\mc{E}_k$ (see e.g. \cite[Theorem 10.34]{lee}), with 
$$
\rank(\ker(L)) = \rank(\mc{E}_k) - \rank(TM) = (d+n) - n = d.
$$
Since clearly $\ker(L) = E_k$, we are done. \end{proof}

\begin{claim} \label{claim:etacont}
    $\eta_1$ and $\eta_2$ are continuous. 
\end{claim}

\begin{proof}
    $\Pi_1$ and $\Pi_2$ can be viewed as continuous bundle maps between metric vector bundles, namely
    \begin{align*}
    &\Pi_1 : E_k \to (\wh{\pi}^k_1)^* T\mc{H}; \\
    &\Pi_2 : E_k \to (\wh{\pi}^k_2)^*(\Omega_k \times \Omega_{n-k}), 
    \end{align*}
    where $\wh{\pi}^k_1 : \mc{G}_k \to \mc{H}$ and $\wh{\pi}^k_2 : \mc{G}_k \to \gr_k(TM) \times \gr_{n-k}(TM)$ denote the natural projections. Thus their normal Jacobians are continuous functions of the base.
\end{proof}

\begin{claim} \label{claim:eta2pos}
    $\eta_2 > 0$ everywhere on $\mc{G}_k$. 
\end{claim}

\begin{proof}
    $\Pi_2(h,\sigma_p,\sigma_q)$ has full rank for all $(h,\sigma_p,\sigma_q) \in \mc{G}_k$.
\end{proof}
\vs
Now define 
$$
\wh{\mc{F}}_{\sigma_p, \sigma_q} := \left( \wh{\pi}^k_2 |_{\ucirc{\mathcal{G}}_k} \right)^{-1}(\sigma_p, \sigma_q) = ( \wh{\pi}^k_2 )^{-1}(\sigma_p, \sigma_q) \cap \ucirc{\mc{G}}_k.
$$ 
Since $\wh{\pi}^k_2|_{\ucirc{\mathcal{G}}_k}$ is a surjective $C^1$ submersion, $\wh{\mc{F}}_{\sigma_p, \sigma_q}$ is an embedded $(d-n)$-dimensional submanifold of $\ucirc{\mc{G}}_k$ for all $(\sigma_p, \sigma_q) \in \gr_k(TM) \times \gr_{n-k}(TM)$. As usual we let $d\wh{\mc{F}}_{\sigma_p, \sigma_q}$ denote the Riemannian density induced by the restriction of the metric on $\mc{G}_k$ to $\wh{\mc{F}}_{\sigma_p, \sigma_q}$. 

\begin{claim} \label{claim:eta1pos}
     For all $(\sigma_p, \sigma_q) \in \gr_k(TM) \times \gr_{n-k}(TM)$, $\eta_1$ is positive somewhere on the fiber $\wh{\mc{F}}_{\sigma_p, \sigma_q}$.
\end{claim}

\begin{proof}
Let $(\sigma_p, \sigma_q) \in \gr_k(TM) \times \gr_{n-k}(TM)$. Clearly it suffices to find $h \in \ucirc{\mc{H}}$ such that $h(p) = q$ and $\Pi_1(h, \sigma_p, \sigma_q)$ is surjective. Observe as in Proposition \ref{prop:pi1critpts} that $\Pi_1(h, \sigma_p, \sigma_q)$ is surjective iff $dh(\sigma_p) \pitchfork \sigma_q$. But by (A2) we can find $h \in \ucirc{\mc{H}}$ such that $dh(\sigma_p) = {\sigma_q}^\perp$, so we are done. 
\end{proof}

\begin{remark}
Claim \ref{claim:eta1pos} is in fact the only place where (A2) is used, and its proof makes clear that this assumption is actually stronger than required for this result. What we really need is the following: 
\begin{itemize}[label=$\operatorname{(A2^*)}$]
\item For any $\sigma_p \in \gr_k(T M)$ and $\sigma_q \in \gr_{n-k}(TM)$, there is an $h \in \ucirc{\mathcal{H}}$ such that $h(p) = q$ and $dh(\sigma_p) \pitchfork \sigma_q$.
\end{itemize}
So for instance (A2), which is equivalent to the claim that for all $\sigma \in \gr_k(TM)$ and $q \in M$ we have
$$
\gr_k(T_q M) \subset \{  dh(\sigma) \ : \ h \in \ucirc{\mathcal{H}}  \}, 
$$
could be replaced by the requirement that for all $\sigma \in \gr_k(TM)$ and $q \in M$, the set 
$$
\gr_k(T_q M) \cap \{  dh(\sigma) \ : \ h \in \ucirc{\mathcal{H}}  \}
$$
has nonempty interior in $\gr_k(T_q M)$.\footnote{This latter property is satisfied, in particular, if we assume that $\ucirc{ev}_0$ is surjective (i.e. the interior of $\mc{H}$ acts transitively on points) and $ev_k$ is a submersion for all $k$.} We prefer (A2) here largely for simplicity, and because, in the general setting we consider, it is no more difficult to construct a family $\mc{H}$ satisfying this hypothesis than $\operatorname{(A2^*)}$; see \S \ref{sec:construction}. The sufficiency of the weaker assumption may, however, be useful in applications where $\operatorname{(A2^*)}$ is easy to verify for a given family $\mc{H}$, despite additional constraints which render (A2) unachievable. 
\end{remark}

The main result of this section is now an immediate consequence of the following general fact. For details of the proof, see Appendix A.  

\renewcommand{\proofname}{Sketch of proof}

\begin{prop}[Appendix A, Lemma \ref{lem:appfiberbounds}] \label{prop:boundsprop} Let $f : M^m \to N^n$ be a $C^1$ submersion between compact Riemannian manifolds, with $\bd M \neq \varnothing$ and $\bd N = \varnothing$, and $\phi \in C^0(M, \RR_{\geq 0})$. Write $\mathcal{F}_q := f^{-1}(q) \cap \ucirc{M}$. Then there exists $C_U^\phi > 0$ such that 
$$
\ds\int_{\mathcal{F}_q} \phi \ d\mathcal{F}_q \leq C_U^\phi
$$
for all $q \in N$.

Suppose in addition that $f\big|_{\ucirc{M}}$ is surjective and $\phi$ is positive somewhere on each interior fiber $\mc{F}_q$. Then there also exists $C_L^\phi > 0$ such that 
$$
C_L^\phi \leq \ds\int_{\mathcal{F}_q} \phi \ d\mathcal{F}_q
$$
for all $q \in N$.

\end{prop}

\begin{proof}
    For the upper bound, it clearly suffices to show that $\vol_{m-n}(\mc{F}_q) \leq \wh{C}_U$ for some uniform $\wh{C}_U < +\infty$; then one can set $C_U^\phi := \wh{C}_U \cdot \norm{\phi}_\infty$. By attaching a collar to $M$ and extending the metric we obtain an open Riemannian manifold $\wh{M}$ isometrically containing $M$; by shrinking the collar we can assume that $f$ extends to a submersion $F : \wh{M} \to N$. Now choose a finite collection $\mc{B}$ of submersion charts for $F$ which cover $M$, such that the plaques of elements of $\mc{B}$ have volumes bounded uniformly above by some $C_1 < +\infty$. Since every fiber $\mc{F}_q$ intersects each submersion chart in at most one plaque, we get $\vol(\mc{F}_q) \leq C_1 \cdot \#(\mc{B})$.

    For the lower bound, one can use surjectivity of $f|_{\ucirc{M}}$ and compactness of $N$ to find finitely many submersion charts $B_1, \dots, B_\ell \subset \ucirc{M}$ for $f|_{\ucirc{M}}$ whose images cover $N$, with fiber plaque volumes bounded below by some $\wh{C}_L > 0$ and $\phi \geq \delta > 0$ on every $B_j$. Then each $\mc{F}_q$ contains at least one plaque of some $B_j$, so
    $$
    \ds\int_{\mc{F}_q} \phi d\mc{F}_q \geq \wh{C}_L \delta > 0, 
    $$
    as desired.
\end{proof}

\renewcommand{\proofname}{Proof}

\begin{cor} \label{cor:fibbounds}
Let $\eta = \frac{\eta_1}{\eta_2} : \mc{G}_k \too \RR_{\geq 0}$. Then there exists $C_k > 1$ such that
$$
\dfrac{1}{C_k} \leq \ds\int_{\wh{\mc{F}}_{\sigma_p, \sigma_q}} \eta \ d\wh{\mc{F}}_{\sigma_p, \sigma_q} \leq C_k
$$
for all $(\sigma_p, \sigma_q) \in \gr_k(TM) \times \gr_{n-k}(TM)$. 
\end{cor}

\begin{proof}
    By Claims \ref{claim:grasssub}, \ref{claim:etacont}, \ref{claim:eta2pos}, and \ref{claim:eta1pos}, this follows from Proposition \ref{prop:boundsprop} with $f = \wh{\pi}^k_2$ and $\phi = \eta$. 
\end{proof}

\begin{remark} 
    If we define 
    $$
    \Phi_{k} : \gr_k(TM) \times \gr_{n-k}(TM) \too \RR_{\geq 0}
    $$
    by 
    $$
    \Phi_k(\sigma_p, \sigma_q) = \ds\int_{\wh{\mc{F}}_{\sigma_p, \sigma_q}} \dfrac{\eta_1}{\eta_2} \ d\wh{\mc{F}}_{\sigma_p, \sigma_q},
    $$
    then Corollary \ref{cor:fibbounds} simply says that $\Phi_k$ is bounded uniformly above and away from zero. Since $\gr_j(TM)$ is compact for all $j$ and $\Phi_k > 0$ by Claim \ref{claim:eta1pos}, this would be obvious if $\Phi_k$ were continuous, and indeed when $\bd \mc{H} = \varnothing$ this is the case by an easy application of Ehresmann's fibration theorem. When $\bd \mc{H} \neq \varnothing$, however, Ehresmann's theorem fails, and in fact $\Phi_k$ cannot be assumed continuous in general. We were led to Proposition \ref{prop:boundsprop} by the need to fill this gap, since it shows in particular that the (non-)continuity of $\Phi_k$ is inessential for the desired bounds.\footnote{An easy adaptation of the proof of the lower bound in Proposition \ref{prop:boundsprop} does, however, show that $\Phi_k$ is always lower semicontinuous.}

    As a final aside, we note that Poincar\'{e}'s formula (as stated in Theorem \ref{thm:poincare}) reduces after the results of the next two sections to the claim that the functions $\Phi_{k}$ (and all the analogously defined functions $\Phi_{k,\ell}$ for $n \leq k+\ell \leq 2n$, where $\Phi_k =: \Phi_{k, n-k}$) are \emph{constant} in the homogeneous case, and moreover positive provided the induced action on each $\gr_k(TM)$ is transitive. This is in turn a fairly straightforward (intuitively almost obvious) consequence of some basic geometry of compact Lie groups. 
\end{remark}

\section{Proof of Main Theorem}

Let $V, W \subset M$ be complementary dimensional submanifolds of $M$, say $\dim(V) = k$ and $\dim(W) = n-k$. Our proof of Theorem \ref{thm:main} now consists of two steps. First, we convert the integral  
$$
\ds\int_\mathcal{H} \#(h(V) \cap W) \ d\mathcal{H}
$$
to an integral over $V \times W$. We then apply the results of \S \ref{sec:bundles} to bound the new integrand uniformly above and away from zero, proving the result.\footnote{For a related use of this ``double fibration" method, see \cite{armentano2022random}.}

\subsection{The co-area formula} \label{sec:co-area}

The key tool in the first step is the co-area formula, which we state here for convenience. Though we choose to give only the simplest form sufficient for our needs, note that this is a special case of a more general result; see e.g. \cite[\S 3.2]{fed}, \cite{nic}, \cite{sim}.

\begin{prop}[The co-area formula, $C^r$ case] \label{prop:co-area}
Let $f : M^{n+k} \to N^n$ be a $C^r$ map between smooth Riemannian manifolds (without boundary), with $r \geq k+1 \geq 1$. For $q \in N$, write $F_q = f^{-1}(q)$. Then for any Borel measurable function $h$ on $M$ we have
$$
\ds\int_M h \cdot \nj(f) \ dM = \ds\int_{q \in N} \left( \ds\int_{F_q} h \ dF_q \right) dN,
$$
with both sides finite if one is. In particular, if $h$ is (essentially) bounded and $\nj(f) \in L^1(M)$, then 
$$
\left( \Phi : q \longmapsto \ds\int_{F_q} h \ dF_q \right) \in L^1(N),
$$
and
$$
\ds\int_M h \cdot \nj(f) \ dM = \ds\int_{N} \Phi \ dN.
$$

In the special case when $f$ is a submersion, the above holds for all $k \geq 0$ assuming only that $f$ is $C^1$. 
\end{prop}

\begin{remark}
Note that the function 
\begin{equation*} 
\Phi : \ N \ni q \longmapsto \ds\int_{F_q} h \ dF_q \in \RR_{\geq 0} \cup \{ \infty \}
\end{equation*}
is well-defined only at regular values $q$ of $f$, for which $F_q$ is an embedded $k$-dimensional submanifold of $M$ with induced Riemannian density $dF_q$. But then by Sard's theorem it is defined a.e. on $N$, so (provided it is measurable) its integral over $N$ makes sense. This application of Sard's is the main reason for the required relation between the regularity of $f$ and $k = \dim(M) - \dim(N)$ and explains why this regularity assumption is unnecessary when $f$ is a submersion, since in that case \emph{every} point of $N$ is a regular value for $f$. Indeed, in the submersive case the co-area formula is an easy consequence of Fubini's theorem. 

More general versions of the co-area formula avoid these issues entirely by replacing $dF_q$ with the $k$-dimensional Hausdorff measure associated to the metric space structure on $M$; see in particular Proposition \ref{prop:co-area2}. These measures agree when $F_q$ is a $k$-dimensional submanifold of $M$.
\end{remark}

We now apply the co-area formula to prove the following, the main result of this section: 

\begin{prop} \label{prop:doubleco}
For $\mathcal{H}, V,$ and $W$ as in Theorem \ref{thm:main} we have
$$
\ds\int_{h \in \mathcal{H}} \#(h(V) \cap W) \ d\mathcal{H} = \ds\int_{(p, q) \in V \times W} \left( \ds\int_{\mathcal{V}_{p,q}} \dfrac{\nj(\pi_1)}{\nj(\pi_2)} \ d\mathcal{V}_{p, q} \right) dV dW,
$$
with both sides finite if one is. 
\end{prop}

\begin{proof} Recall the diagram
\begin{center}
\begin{tikzcd}
& \arrow{ld}[above]{\pi_1 \ \ \ } \mathcal{V} \arrow[rd, "\pi_2"] & \\
\ucirc{\mathcal{H}} & & V \times W
\end{tikzcd}
\end{center}
of Riemannian manifolds and $C^1$ maps. Since
$$
\dim(\mathcal{V}) = \dim(\mathcal{H}) = d
$$
and $\pi_1$ is $C^1$, we may apply Proposition \ref{prop:co-area} to $\pi_1$ (with $h \equiv 1$) to obtain
\begin{equation} \label{eqn:co1}
\ds\int_{\mathcal{V}} \ \nj(\pi_1) \ d\mathcal{V} = \ds\int_{h \in \ucirc{\mathcal{H}}} \left( \ds\int_{\pi_1^{-1}(h)} d\pi_1^{-1}(h) \right) d\mathcal{H} = \ds\int_{h \in \mathcal{H}} \#(h(V) \cap W) \ d\mathcal{H}.
\end{equation}
The second equality follows from the fact that $\pi_1^{-1}(h)$ is a $0$-dimensional submanifold of $\mathcal{V}$ for a.e. $h \in \mathcal{H}$, i.e. a discrete set of points, in which case we have
$$
\ds\int_{\pi_1^{-1}(h)} d\pi_1^{-1}(h) = \ds\sum_{x \in \pi_1^{-1}(h)} 1 = \#(\pi_1^{-1}(h)) = \#(h(V) \cap W).
$$
Note in particular that the a.e. defined function 
$$
h \longmapsto \ds\int_{\pi_1^{-1}(h)} d\pi_1^{-1}(h) = \#(h(V) \cap W)
$$
will be in $L^1(\ucirc{\mathcal{H}}) = L^1(\mathcal{H})$ provided $\nj(\pi_1) \in L^1(\mathcal{V})$.

We now apply the co-area formula a second time, this time to the map $\pi_2$ with $h = \frac{\nj(\pi_1)}{\nj(\pi_2)}$. Observe that this is well-defined because $\pi_2$ is a submersion, so $\nj(\pi_2) > 0$ everywhere on $\mathcal{V}$. We get
\begin{align*}
  \ds\int_{\mathcal{V}} \ \nj(\pi_1) \ d\mathcal{V} &= \ds\int_{\mathcal{V}} \ \left( \dfrac{\nj(\pi_1)}{\nj(\pi_2)} \right) \cdot \nj(\pi_2) \ d\mathcal{V} \\
  &= \ds\int_{(p, q) \in V \times W} \left( \ds\int_{\mathcal{V}_{p,q}} \dfrac{\nj(\pi_1)}{\nj(\pi_2)} \ d\mathcal{V}_{p,q} \right) dV dW. \\
\end{align*}
\noindent Putting this together with (\ref{eqn:co1}) gives the desired result. \end{proof}

\subsection{Fiber integral bounds} \label{sec:fbounds}

By the preceding section, it now suffices to give uniform bounds for the integrals
$$
\ds\int_{\mathcal{V}_{p,q}} \dfrac{\nj(\pi_1)}{\nj(\pi_2)} \ d\mathcal{V}_{p, q}.
$$

\noindent Most of the necessary work has been done in \S \ref{sec:bundles}. First recall Lemma \ref{lem:etatonj}, which gives
    \begin{align*}
    T_{(h,p,q)} \mc{V} &= E_k(h,T_p V, T_q W), \\[2mm]
    d(\pi_i)_{(h, p, q)} &= \Pi_i(h, T_p V, T_q W), \\[2mm]
    \dfrac{\nj(\pi_1)}{\nj(\pi_2)}(h,p,q) = \dfrac{\eta_1}{\eta_2}(h, T_p &V, T_q W) = \eta(h, T_p V, T_q W).
    \end{align*}
\noindent Here $i=1,2$ and $E_k,$ $\Pi_i$, $\eta_i$ are as defined in \S \ref{sec:bundles}. As an immediate consequence we obtain

\begin{lem} \label{lem:isomfib}
For all $(p,q) \in V \times W$, the fiber $\mathcal{V}_{p, q}$ is isometric to $\wh{\mc{F}}_{T_p V, T_q W}$. Moreover, we have
$$
\ds\int_{\wh{\mc{F}}_{T_p V, T_q W}} \eta \ d\wh{\mc{F}}_{T_p V, T_q W} = \ds\int_{\mathcal{V}_{p,q}} \dfrac{\nj(\pi_1)}{\nj(\pi_2)} \ d\mathcal{V}_{p, q}.
$$
\end{lem}

\begin{proof}
    Fix $(p,q) \in V \times W$. Using the fact that the Riemannian metrics on $\mc{G}_k$ and $\mc{V}$ are simply the restrictions of the product metrics on $\mc{H} \times \gr_k(TM) \times \gr_{n-k}(TM)$ and $\mc{H} \times M \times M$, respectively, it is easy to see that the map 
    $$
    \wh{\mc{F}}_{T_p V, T_q W} \too \mc{V}_{p,q}, \ \ (h, T_p V, T_q W) \mapstoo (h, p, q)
    $$ 
    is an isometry.\footnote{Put more intuitively, both sides are just the set of $h \in \ucirc{\mc{H}}$ such that $h(p) = q$ as a submanifold of $\mc{H}$ with the induced Riemannian metric.} The fact that
    $$
    \ds\int_{\mathcal{V}_{p,q}} \dfrac{\nj(\pi_1)}{\nj(\pi_2)}(h,p,q) \ d\mathcal{V}_{p, q}
    = \ds\int_{\wh{\mc{F}}_{T_p V, T_q W}} \eta(h, T_p V, T_q W) \ d\wh{\mc{F}}_{T_p V, T_q W}
    $$
    then follows by change of variables.
\end{proof}

\begin{cor} \label{cor:fibbds}
    If $C_k \geq 1$ is the constant given by Corollary \ref{cor:fibbounds}, then
    $$
    \dfrac{1}{C_k} \leq \ds\int_{\mathcal{V}_{p,q}} \dfrac{\nj(\pi_1)}{\nj(\pi_2)} \ d\mathcal{V}_{p, q} \leq C_k
    $$
    for all $p, q \in V \times W$. 
\end{cor}

\noindent Combining this with Proposition \ref{prop:doubleco} and setting $C = \ds\max_{0 \leq k \leq n} C_k$ completes the proof of Theorem \ref{thm:main}.  

\begin{remark}
    Our use of the rather abstract results of \S \ref{sec:bundles} and Appendix A to obtain the constants $C_k$ leaves their geometric meaning somewhat opaque. The advantage of this approach is that it generalizes easily to prove Theorem \ref{thm:generalization}, but in the complementary-dimensional case a more concrete picture is available, which may be useful in applications where estimates of $C$ become desirable. We sketch this approach in Appendix B.
\end{remark}

\section{Generalization} \label{sec:gen}

We now describe the generalization of Theorem \ref{thm:main} to the case when $\dim(V) + \dim(W) \geq \dim(M)$, where $\#(h(V) \cap W)$ is replaced by the appropriate-dimensional Hausdorff measure associated to the metric space structure on $M$.

For $r \geq 0$ and any Riemannian manifold $X$, let $H_X^r$ denote the $r$-dimensional Hausdorff measure on $X$ associated to its induced metric space structure. When $r = \dim(X)$, $H_X^r$ agrees with the volume measure $\nu_X$ induced by the metric on $X$. More generally, if $Y \subset X$ is a $k$-dimensional submanifold with Riemannian metric inherited from $X$ and induced volume measure $\nu_Y$, we have
$$
H_X^k \big|_Y = H_Y^k = \nu_Y.
$$
In particular, the volume of a $k$-dimensional submanifold of $X$ is equal to its $k$-dimensional Hausdorff measure in $X$. In this sense, for $k \in \NN$ one can think of $H_X^k$ as a generalization of $k$-dimensional volume to non-smooth subsets of $X$. 

We now consider a fixed closed Riemannian manifold $M$ of dimension $n$. For simplicity we write $H^r := H^r_M$.

\begin{thm} \label{thm:generalization}

Suppose $(\mathcal{H}, \psi)$ is compact $C^1$ family of diffeomorphisms of $M$ such that
    \vspace{2mm}
\begin{enumerate}[label=$\operatorname{(A\arabic*)}$]
    \item $ev : \mathcal{H} \times M \to M \times M$ defined by $ev(h,p) = (p, h(p))$ is a submersion;
    
    \vspace{3mm}
    
    \item $\ucirc{ev}_k : \ucirc{\mathcal{H}} \times \gr_k(TM) \to \gr_k(TM) \times \gr_k(TM)$ defined by $ev_k(h, \sigma) = (\sigma, dh(\sigma))$ is surjective for all $0 \leq k \leq n-1$.

    \end{enumerate}

    \vspace{2mm}

\noindent Then there is a constant $C = C(\mathcal{H}, \psi) \geq 1$ such that for any submanifolds $V^k, W^\ell \subset M$ with $k + \ell \geq n$, we have

\begin{equation} \label{eqn:genbounds}
\dfrac1C \vol(V)\vol(W) \ \ \leq \ \ \ds\int_{h \in \mathcal{H}} H^{k+\ell-n}(h(V) \cap W) \ d\mathcal{H} \ \ \leq \ \ C \vol(V)\vol(W). \\
\end{equation}
\end{thm}

The need to work with Hausdorff measures rather than volumes here derives from our insistence on assuming only $C^1$ regularity. In this context, when $\dim(V) + \dim(W) > \dim(M)$ we can no longer assume that $h(V) \pitchfork W$ for a.e. $h \in \mc{H}$; thus $\vol(h(V) \cap W)$ is not necessarily a.e. well-defined. However, if we let $h(V) \cap_T W$ denote the set of transversal intersection points between $h(V)$ and $W$ (which when nonempty is always a $(k+\ell-n)$-dimensional submanifold of $M$), we will see below that indeed
$$
H^{k+\ell-m}(h(V) \cap W) = H^{k+\ell-m}(h(V) \cap_T W) = \vol(h(V) \cap_T W)
$$
for a.e. $h \in \mathcal{H}$.

Note that, since $H^0$ is the counting measure on $M$, Theorem \ref{thm:generalization} does indeed include Theorem \ref{thm:main} as a special case. Moreover, the proof method is essentially the same. Thus in this section we will content ourselves with some remarks on the modifications needed to adapt our proof of Theorem \ref{thm:main} to the general setting. 

Let $V^k, W^\ell \subset M$ be submanifolds with $k + \ell \geq n$. The solution manifold $\widetilde{\mathcal{V}} \subset \mathcal{H} \times M \times M$ remains defined as before, and again (A1) implies that
$$
\mathcal{V} := \left( {\tilde{\pi}_2}|_{\text{int}(\widetilde{\mathcal{V}})} \right)^{-1}(V \times W)
$$
is a $C^1$ submanifold of $\widetilde{\mathcal{V}}$. Now, however, we have
$$
\dim(\mathcal{V}) = (d + n) - \codim(V \times W)  = (d+n) - (2n - k - \ell) = d + k + \ell - n. 
$$

\vspace{1mm}

The basic observation motivating these definitions and underlying our entire proof strategy for Theorem \ref{thm:main} was that
$$
\#(h(V) \cap W) = \#(\pi_1^{-1}(h)),
$$
or equivalently, 
$$
H^0_M(h(V) \cap W) = H^0_\mc{V}(\pi_1^{-1}(h)).
$$
The reason was that there is an obvious bijection between these sets (given explicitly by $\pi_3 : \mc{V} \to M$, defined below), and bijections between metric spaces preserve $0$-dimensional Hausdorff measure. Obviously this fails for higher dimensional Hausdorff measures, so for the general case we must attend to how the maps $\pi_3^h := \pi_3|_{\pi_1^{-1}(h)} : \pi_1^{-1}(h) \to h(V) \cap W$ transform $(k+\ell-n)$-dimensional Hausdorff measures, where $\pi_3 : \mc{V} \to W$ is defined as one would expect, by $\pi_3(h,p,q) = q$.\footnote{In fact this issue is largely an artifact of our original choice of set-up: it can be avoided by replacing $\wt{\mc{V}}$ with $\mc{H} \times M$ and $\wt{\pi}_2 : \wt{\mc{V}} \to M \times M$ by the map $\mc{H} \times M \to M \times M$, $(h, q) \mapsto (h^{-1}(q), q)$, since for these choices the resulting identification of $\pi_1^{-1}(h)$ with $h(V) \cap W$ is isometric. See \cite{condition} for an example of this approach.}

First suppose that $h$ is a regular value of $\pi_1$; this is true iff $h(V) \pitchfork W$. In this case, $\pi_1^{-1}(h)$ and $h(V) \cap W$ are $(k+\ell-n)$-dimensional submanifolds of $\mathcal{V}$ and $M$, respectively, and $\pi_3^h$ is a $C^1$ diffeomorphism. Their $(k+\ell-n)$-Hausdorff measures then agree with the volume measures induced by their Riemannian structures, and moreover we have
\begin{align*}
H^{k+\ell-n}_M(h(V) \cap W) = \vol(h(V) \cap W) &= \vol(\pi_3(\pi_1^{-1}(h))) \\
&= \ds\int_{\pi_1^{-1}(h)} \norm{\textstyle\bigwedge^{k+\ell-n}d\pi_3^h} d\pi_1^{-1}(h) \\
&= \ds\int_{\pi_1^{-1}(h)} \norm{\textstyle\bigwedge^{k+\ell-n}d\pi_3^h} dH^{k+\ell-n}_\mathcal{V}.
\end{align*}
Note also that 
$$
d\pi^h_{3}(h,p,q) = d\pi_3|_{T_{(h,p,q)}{\pi_1^{-1}(h)}} = d\pi_3|_{\ker(d\pi_{1}(h,p,q))},
$$
and this last expression is well-defined even when $(h,p,q)$ is not a regular point for $\pi_1$. Thus we can define
$$
\rho : \mc{V} \longrightarrow \RR_{\geq 0}
$$
by
$$
\rho(h,p,q) = \norm{\textstyle\bigwedge^{k+\ell-n}d\pi_3\big|_{\ker(d\pi_{1}(h,p,q))}},
$$
and the above equation becomes
\begin{equation} \label{eqn:jacgen} 
H^{k+\ell-n}_M(h(V) \cap W) = \ds\int_{\pi_1^{-1}(h)} \rho \ dH^{k+\ell-n}_\mathcal{V}.
\end{equation}

 So far, (\ref{eqn:jacgen}) holds only when $h$ is a regular value of $\pi_1$. In the complementary dimensional case this was enough, since when $\dim(\mc{H}) = \dim(\mc{V})$ Sard's theorem applies and thus (\ref{eqn:jacgen}) holds for a.e. $h \in \mc{H}$. When $\dim(V) + \dim(W) > \dim(M)$, however, we have $\dim(\mc{V}) > \dim(\mc{H})$, so Sard's no longer applies (since $\pi_1$ can only be assumed $C^1$). We appeal instead to the following generalization: 

\begin{prop}[\cite{sim}, Theorem 6.4] \label{prop:sardsgen}
Let $f: M^m \to N^n$ be a $C^1$ map between Riemannian manifolds, and let $C \subset M$ denote the set of critical points of $f$. Set $r = \max\{0, m-n\}$. Then
\begin{equation}
    H^r_M(f^{-1}(y) \cap C) = 0
\end{equation}
for Lebesgue a.e. $y \in N$.
\end{prop}

\begin{remark}
    The interesting case (and the one relevant here) is when $m > n$. As Simon notes in \cite{sim}, this proposition is itself a straightfoward consequence of Federer's general co-area formula; see \cite[\S 3.2]{fed} for details. 
\end{remark}

To apply this to the present setting, first write $C \subset \mathcal{V}$ for the set of critical points of $\pi_1$, and for $h \in \mathcal{H}$, let $h(V) \cap_T W$ denote the set of transversal intersection points between $h(V)$ and $W$, and $h(V) \cap_{NT} W = (h(V) \cap W) \setminus (h(V) \cap_T W)$. Then we have
\begin{align*}
    \pi_3(\pi_1^{-1}(h) \setminus C) &= h(V) \cap_T W, \\
    \pi_3(\pi_1^{-1}(h) \cap C) &= h(V) \cap_{NT} W, 
\end{align*}
and 
\begin{align*}
H^{k+\ell-n}_M(h(V) \cap W) &= H^{k+\ell-n}_M(h(V) \cap_T W) + H^{k+\ell-n}_M(h(V) \cap_{NT} W) \\
&= \ds\int_{\pi_1^{-1}(h) \setminus C} \rho \ dH^{k+\ell-n}_\mathcal{V} + H^{k+\ell-n}_M(\pi_3(\pi_1^{-1}(h) \cap C)).
\end{align*}
By Proposition \ref{prop:sardsgen}, 
$$
\ds\int_{\pi_1^{-1}(h) \setminus C} \rho \ dH^{k+\ell-n}_\mathcal{V} = \ds\int_{\pi_1^{-1}(h)} \rho \ dH^{k+\ell-n}_\mathcal{V}
$$
for a.e. $h \in \mathcal{H}$. For the second term, observe that $\pi_3 = \widetilde{\pi}_3|_{\pi_1^{-1}(h)}$, where $\widetilde{\pi}_3 : \widetilde{\pi}_1^{-1}(h) \to M$ is a distance non-increasing diffeomorphism (indeed $\widetilde{\pi}_1^{-1}(h)$ is canonically identified with $\Gamma(h) \subset M \times M$, with $\widetilde{\pi}_3$ corresponding to projection onto the second coordinate), so that
$$
H^{k+\ell-n}_M(\pi_3(\pi_1^{-1}(h) \cap C)) \leq H^{k+\ell-n}_\mathcal{V}(\pi_1^{-1}(h) \cap C).
$$
Since the right-hand side is zero for a.e. $h \in \mc{H}$ by Proposition \ref{prop:sardsgen}, so is the left. We conclude that
\begin{align*}
H^{k+\ell-n}_M(h(V) \cap W) &= \ds\int_{\pi_1^{-1}(h) \setminus C} \rho \ dH^{k+\ell-n}_\mathcal{V} + H^{k+\ell-n}_M(\pi_3(\pi_1^{-1}(h) \cap C)) \\
&= \ds\int_{\pi_1^{-1}(h)} \rho \ dH^{k+\ell-n}_\mathcal{V}
\end{align*}
for Lebesgue a.e. $h \in \mc{H}$.

To apply this formula, we need the following slightly more general form of the co-area formula: 

\begin{prop}[The co-area formula, $C^1$ case] \label{prop:co-area2}
Let $f : M^m \to N^n$ be a $C^1$ map between Riemannian manifolds, $m \geq n$. For $q \in N$, write $F_q = f^{-1}(q)$. Then for any Borel measurable function $h$ on $M$ we have
$$
\ds\int_M h \cdot \nj(f) \ dM = \ds\int_{q \in N} \left( \ds\int_{F_q} h \ dH^{m-n}_N \right) dN,
$$
with both sides finite if one is.
\end{prop}

We can now proceed as before. First, apply Proposition \ref{prop:co-area2} twice to deduce
\begin{align*}
    \ds\int_{h \in \mathcal{H}} H^{k+\ell-n}(h(V) \cap W) \ d\mathcal{H} &= \ds\int_{h \in \mathcal{H}} \left( \ds\int_{\pi_1^{-1}(h)} \rho \ dH^{k+\ell-n}_\mathcal{V} \right) \ d\mc{H} \\
    &= \ds\int_\mc{V} \rho \cdot \nj(\pi_1) \ d\mc{V} \\
    &= \ds\int_{(p,q) \in V \times W} \left( \ds\int_{\mc{V}_{p,q}} \dfrac{\rho \cdot \nj(\pi_1)}{\nj(\pi_2)} \ d\mc{V}_{p,q} \right) dV dW,
\end{align*}
where $\mc{V}_{p,q} = \pi_2^{-1}(p,q)$ as before. The proof now reduces to showing that the integrand $\frac{\rho \cdot \nj(\pi_1)}{\nj(\pi_2)}$ can be viewed as a continuous function on the appropriate Grassmannian bundle $\mc{G}_{k,\ell}$ over $\widetilde{\mathcal{V}}$, which is positive somewhere on each fiber of the submersion $\ucirc{\mathcal{G}}_{k,\ell} \to \gr_k(TM) \times \gr_\ell(TM).$ 

The space $\mc{G}_{k, \ell} \subset \mc{H} \times \gr_k(TM) \times \gr_\ell(TM)$, the bundle $E_{k,\ell}$, and the functions $\eta_1, \eta_2$, and $\eta$ are defined exactly as in \S \ref{sec:bundles} (so e.g. in our original notation, $\mc{G}_k = \mc{G}_{k, n-k}$). Again, $\eta : \mc{G}_{k, \ell} \to \RR_{\geq 0}$ is continuous, and positive at $(h, \sigma_p, \sigma_q)$ iff $dh(\sigma_p) + \sigma_q = T_q M$. To extend $\rho$ to a function on $\mc{G}_{k, \ell}$, simply define 
\begin{align*}
\Pi_3(h, \sigma_p, \sigma_q) : E_{k, \ell}(h, \sigma_p, \sigma_q) &\too T_q M \\
(\dot{h}, v, w) &\mapstoo w,
\end{align*}
and $\wh{\rho} : \mc{G}_{k,\ell} \to \RR_{\geq 0}$ by
$$
\wh{\rho}(h, \sigma_p, \sigma_q) = \norm{\textstyle\bigwedge^{k+\ell-n}\Pi_3(h, \sigma_p, \sigma_q)\big|_{\ker(\Pi_1(h, \sigma_p, \sigma_q))}}.
$$
Clearly we have
$$
\wh{\rho}(h, T_p V, T_q W) = \rho(h, p, q).
$$
Moreover, $\wh{\rho}$ is continuous at all points where $\Pi_1$ has full rank, that is, away from the zeros of $\eta$. Since also $0 < \wh{\rho} \leq 1$, it follows that $\wh{\rho} \cdot \eta$ is continuous on $\mc{G}_{k, \ell}$, with the same zero set as $\eta$, and satisfies
$$
(\wh{\rho} \cdot \eta)(h, T_p V, T_q W) = \left( \rho \cdot \dfrac{NJ(\pi_1)}{NJ(\pi_2)} \right) (h, p, q). 
$$
One can now apply the results of Appendix A as before to complete the proof.

\section{Construction of $(\mc{H}, \psi)$} \label{sec:construction}

In this section we show that Theorems \ref{thm:main} and \ref{thm:generalization} are non-vacuous, by constructing, for an arbitrary closed manifold $M$, a compact $C^1$ family of diffeomorphisms of $M$ satisfying hypotheses (A1) and (A2). This proves Theorem \ref{thm:existence1}. 

Actually, since it is no more difficult (and arguably more natural) to work in the smooth category here, the family we construct will be $C^\infty$. In addition, we will be unconcerned with compactness until the end. Specifically, we first obtain a $C^\infty$ family $(\RR^N, \psi)$ of the form
$$
\psi : (\RR^N, 0) \xrightarrow[ \ \ \ ]{} (\Diff^\infty(M), \id)
$$
which satisfies hypothesis (A1) and (A2), such that the surjectivity requirements of (A2) are satisfied by $(\BB^N_R, \psi|_{\BB^N_R})$ for all $R$ sufficiently large.\footnote{Here and below we write $\BB^N_R, \overline{\BB}^N_R \subset \RR^N$ for the open and closed balls, respectively, centered at $0$ of radius $R$, with the subscript occasionally omitted in the special case $R = 1$, i.e. $\BB^N := \BB^N_1$.} The desired compact family is then obtained by restricting $\psi$ to $\overline{\BB}^N_R$ for some large $R$.

We begin with some general lemmas. Recall that for $1 \leq r \leq \infty$, a \emph{$C^r$ family of diffeomorphisms of $M$} is a pair $(\mc{H}, \psi)$, with $\mc{H}$ a smooth manifold, possibly with boundary, and $\psi : \mc{H} \to \Diff^r(M)$ such that the associated evaluation map $ev_\psi : \mc{H} \to M$ is $C^r$; this implies in particular that $\psi$ is continuous. Any $C^r$ family of diffeomorphisms of $M$ is, of course, a $C^s$ family of diffeomorphisms for all $s < r$. 

One motivation for interpreting smoothness in this way---instead of via the strictly stronger requirement that $\psi \in C^r(\mc{H}, \Diffr(M))$ for the standard Banach/Fr\'{e}chet manifold structure on $\Diffr(M)$---is the following elementary fact, which shows that the class of $C^r$ families of diffeomorphisms in our weaker sense is in a natural sense closed under composition.\footnote{That this fails for the stronger notion is a consequence of the well-known fact that the composition map $\left[ \Diffr(M) \right]^2 \to \Diffr(M)$ is only $C^0$ when $r < \infty$.}

\begin{lem} \label{lem:compCr}
Let $(\mathcal{H}_1, \psi_1)$ and $(\mathcal{H}_2, \psi_2)$ be $C^r$ families of diffeomorphisms of $M$, with $\bd \mc{H}_1 = \bd \mc{H}_2 = \varnothing$. Define
$$
\psi :  \mathcal{H}_1 \times \mc{H}_2 \xrightarrow[ \ \ \ ]{} \Diff^r(M)
$$
by $\psi(h_1, h_2) = \psi_2(h_2) \circ \psi_1(h_1)$, i.e. $\psi = \operatorname{comp} \circ (\psi_1 \times \psi_2)$.\footnote{For notational convenience we adopt the convention which defines the composition map $\left[ \Diffr(M) \right]^k \to \Diffr(M)$ by $\operatorname{comp}(f_1, \dots, f_k) = f_k \circ \cdots \circ f_1$.} Then $(\mc{H}_1 \times \mc{H}_2, \psi)$ is a $C^r$ family of diffeomorphisms of $M$. 
\end{lem}

\begin{proof}
It suffices to check that the evaluation map
\begin{align*}
ev_\psi : (\mc{H}_1 \times \mc{H}_2) \times M &\xrightarrow[ \ \ \ \ \ \ ]{} M \\
(h_1, h_2, p) &\xmapsto[ \ \ \ \ \ \ ]{} \psi(h_1, h_2)(p)
\end{align*}
is $C^r$. We have
\begin{align*}
    ev_\psi(h_1, h_2, p) &= \psi(h_1, h_2)(p) \\
    &= \psi_1(h_1) (\psi_2(h_2)(p)) \\
    &= ev_{\psi_1}(h_1, \psi_2(h_2)(p)) = ev_{\psi_1}(h_1, ev_{\psi_2}(h_2, p)).
\end{align*}
Since $ev_{\psi_1}$ and $ev_{\psi_2}$ are both $C^r$ by assumption, so is $ev_\psi$. \end{proof}

This lemma means that we can assemble our $C^\infty$ family in the obvious way, namely by composing a finite collection of smaller $C^\infty$ families supported in (and defined via) charts. Moreover, working within charts it is easy to produce families which \emph{locally} satisfy the desired properties (A1) and (A2); see Proposition \ref{prop:localcon} below. The next lemma gives us a criterion for ensuring that a composition of such families satisfies (A1) globally.

\begin{lem} \label{lem:compsub}
    Let $(\mathcal{H}_1, \psi_1), \dots, (\mathcal{H}_k, \psi_k)$ be $C^r$ families of diffeomorphisms of $M$, with $\bd \mc{H}_i = \varnothing$ for all $1 \leq i \leq k$. Suppose that for all $i = 1, \dots, k$, $h_i \in \mathcal{H}_i$, and $p \in M$, we have
    $$
    \psi_i(h_i)(p) \neq p  \ \ \Longrightarrow \ \ d(ev_{\psi_i, p})_{h_i} \text{ is surjective, }
    $$
    where $ev_{\psi_i, p} := ev_{\psi_i}(\sbullet, p) : \mathcal{H}_i \to M$. Finally, suppose that for all $p \in M$, there exists $j$ such that $ev_{\psi_j, p}$ is a submersion, i.e. such that $d(ev_{\psi_j, p})_{h_j}$ is surjective for all $h_j \in \mc{H}_j$. Then letting $\mc{H} := \mc{H}_1 \times \cdots \mc{H}_k$, the map
    \begin{align*}
    \psi : \mc{H} &\longrightarrow \Diff^r(M) \\
    (h_1, \dots, h_k) &\longmapsto \psi_k(h_k) \circ \cdots \circ \psi_1(h_1)
    \end{align*}
    defines a $C^r$ family of diffeomorphisms of $M$, and $ev_{\psi, p} := ev_\psi(\cdot, p) : \mc{H} \to M$ is a submersion for all $p \in M$.
\end{lem}

\begin{proof}
    The fact that $(\mc{H}, \psi)$ is a $C^r$ family of diffeomorphisms of $M$ follows inductively from Lemma \ref{lem:compCr}. Now let $p \in M$ and $h = (h_1, \dots, h_k) \in \mathcal{H}$. We claim that 
    $$
    d(ev_{\psi, p})_h : T_h \mc{H}  
    \xrightarrow[]{ \ \ \ \ } T_q M
    $$
    is surjective, where $q = ev_{\psi,p}(h) = \psi(h)(p) = \psi_k(h_k) \circ \cdots \circ \psi_1(h_1)(p).$ Clearly it suffices to show that 
    $$
    \partial_{h_j} (ev_{\psi, p})_h : T_{h_j} \mc{H}_j \xrightarrow[]{ \ \ \ \ } T_q M
    $$
    is surjective for some $1 \leq j \leq k$. 

    To compute these derivatives, write $p_0 := p$ and $p_i = \psi_i(h_i) \circ \cdots \circ \psi_1(h_1)(p)$ for $i \geq 1$; in particular $p_k = q$. Setting $F_j = \psi_k(h_k) \circ \cdots \circ \psi_{j+1}(h_{j+1})$, we then have
    \begin{align*}
        \partial_{h_j}(ev_{\psi,p})_h = (dF_j)_{p_j} \circ d(ev_{\psi_j, p_{j-1}})_{h_j}.
    \end{align*}
    
    Now set $\ell = \min \ \{ 1 \leq i \leq k \ | \ d(ev_{\psi_i, p})_{h_i} \text{ is surjective} \}$. Then by our assumption we must have $\psi_i(h_i)(p) = p$, and thus inductively $p_i = p$, for all $i < \ell$. Applying the above formula gives
    \begin{align*}
        \partial_{h_\ell}(ev_{\psi,p})_h = (dF_\ell)_{p_\ell} \circ d(ev_{\psi_\ell, p_{\ell-1}})_{h_j} = (dF_\ell)_{p_\ell} \circ d(ev_{\psi_\ell, p})_{h_\ell}. 
    \end{align*}
    By the definition of $\ell$, $d(ev_{\psi_\ell, p})_{h_\ell}$ is surjective; since $F_\ell$ is a diffeomorphism, it follows that $\partial_{h_\ell}(ev_{\psi,p})_h : T_{h_\ell} \mc{H}_\ell \to T_{p_k} M$ is surjective as well, so we are done. 
\end{proof}

We can now proceed with the construction of $\mathcal{H}$. Our approach is the ``obvious" one, beginning with the observation that $\text{Isom}^+(\RR^n) \cong SO(n,\RR) \ltimes \RR^n$ acts smoothly and transitively on $\gr_k(T\RR^n) \cong \gr_k(\RR^n) \times \RR^n$ for all $0 \leq k \leq n-1$. In particular, for every $\sigma \in \gr_k(T\RR^n)$, $ev_\sigma : SO(n, \RR) \times \RR^n \to \gr_k(T\RR^n)$ is a surjective submersion. To realize these properties inside an arbitrary manifold $M$, we first ``cut off" this action to a smooth family of (uniformly) compactly supported diffeomorphisms of $\RR^n$, which satisfies the submersion property (A1) on the interior of its support, and ``acts transitively" on the restrictions of the Grassmann bundles to some slightly smaller ball. Using charts this family can be be realized inside any coordinate ball in $M$ by a locally supported family of diffeomorphisms. The lemmas above then allow us to obtain the desired ``action" of $\RR^N$ on $M$ by composing sufficiently many such families.

 \subsection{Step 1: Local construction.} 

 Choose $\beta \in C^\infty(\RR_{\geq 0}, [0, 1])$ such that 
\begin{itemize}
    \item $\beta|_{[0,2]} \equiv 1$
    \item $\beta|_{[3,+\infty)} \equiv 0$
    \item $\beta'(x) < 0$ for all $x \in (2, 3)$.
\end{itemize}

Define
 $$
 L : \mathfrak{so}_n \longrightarrow \Diff_c^\infty(\mathbb{R}^n, 0)
 $$
 by
 $$
 L(v)(p) = \exp(\beta(|p|)v)(p),
 $$
 where $\exp : \mathfrak{so}_n \to SO(n,\RR)$ denotes the Lie exponential. Thus $L(v)$ preserves each sphere $\mathbb{S}^{n-1}_r := \{ x \in \RR^n \ : \ |x| = r \}$, and acts on $\mathbb{S}^{n-1}_r$ by $\exp(\beta(r)v) \in SO(n, \RR)$. In particular, $L(v)|_{\BB^n_2} = \exp(v)$, and $\supp(L(v)) = \overline{\BB}^n_3$ for all nonzero $v \in \mathfrak{so}_n$. 

 Define smooth vector fields $V_1, \dots, V_n \in \Gamma^\infty_c(\RR^n) = C^\infty_c(\RR^n, \RR^n)$ by 
$$
V_i(p) = \beta(|p|)\mathbf{e}_i, \ 1 \leq i \leq n.
$$
For each $i$ let $h_i^t$ denote the flow associated to $V_i$. By composing these flows, we obtain a map
\begin{align*}
\tau : \RR^n &\longrightarrow \Diff_c^\infty(\RR^n) \\
(t_1, \dots, t_n) &\longmapsto h_n^{t_n} \circ \cdots \circ h_1^{t_1}.
\end{align*}
Finally, we define 
$$
\psi : \RR^n \times \mathfrak{so}_n \longrightarrow \Diff_c^\infty(\RR^n)
$$
by
$$
\psi(t, v) = \tau(t) \circ L(v). 
$$ 

\begin{prop} \label{prop:localcon}
    The map
    \begin{align*}
        ev_{\psi} : (\RR^n \times \mathfrak{so}_n) \times \RR^n &\longrightarrow \RR^n \\
        (t, v, p) &\longmapsto \psi(t,v)(p)
    \end{align*}
    is $C^\infty$, and for all $p \in \BB^n_3$, the map $ev_{\psi, p} = ev_\psi(\cdot, p)$ is a submersion. Moreover, for all $0 \leq k \leq n-1$ and $\sigma_p, \sigma_q \in \gr_k(T\BB^n)$, we have
    $$
    d(\psi(t,v))(\sigma_p) = \sigma_q.
    $$
    for some $(t, v) \in \BB^n_2 \times B_{\mathfrak{so}_n}(0,R)$, where $R > 0$ is such that $\exp(B_{\mathfrak{so}_n}(0,R)) = SO(n,\RR)$.
\end{prop}

\begin{proof}
    To see that $ev_\psi$ is smooth, first observe that smoothness of 
    \begin{align*}
        ev_{L} : \mathfrak{so}_n \times \RR^n &\longrightarrow \RR^n \\
        (v, p) &\longmapsto L(v)(p) = \exp(\beta(|p|)v)(p)
    \end{align*}
    follows from that of the Lie exponential $\exp$ and $\beta$. Also, since each $V_i$ is smooth,
    \begin{align*}
        ev_{i} : \RR \times \RR^n &\longrightarrow \RR^n \\
        (t, p) &\longmapsto h_i^t(p)
    \end{align*}
    is smooth for all $1 \leq i \leq n$. Smoothness of $ev_\psi$ then follows inductively from Lemma \ref{lem:compCr}.

    Now let $p \in \BB^n_3$ and $(t, v) \in \RR^n \times \mathfrak{so}_n$. Writing $t = (t_1, \dots, t_n)$, we set $q_0 = L(v)(p)$ and
    $$
    q_j = ( h_j^{t_j} \circ \cdots \circ h_1^{t_1} )(q)
    $$
    for $1 \leq j \leq n$; in particular, $q_n = \psi(t,v)(p)$. Observe that since every $L(v)$ and $h_j^{t_j}$ preserves $\BB^n_3$, $q_i \in \BB^n_3$ for all $i$. 

    We would like to show that that 
    $$
    d(ev_{\psi,p})_{(t,v)} : T_t \RR^n \times T_v \mathfrak{so}_n \longrightarrow T_{q_n} \RR^n
    $$
    is surjective. Clearly it suffices to show that this is true for 
    $$
    \partial_t(ev_{\psi,p})_{(t,v)} : T_t \RR^n \longrightarrow T_{q_n} \RR^n.
    $$
    Fix $1 \leq j \leq n$. We compute
    \begin{align*}
    \partial_{t_j}( \psi(t,v)(p) ) = \partial_{t_j}( \tau(t)(q_0) )
    &= \partial_{t_j} ( h_n^{t_n} \circ \cdots \circ h_1^{t_1}(q_0) ) \\
    &= \partial_{t_j} ( h_n^{t_n} \circ \cdots \circ h_j^{t_j}(q_{j-1}) ) \\
    &= d(h_n^{t_n} \circ \cdots \circ h_{j+1}^{t_{j+1}} ) \left( \partial_{t_j}(h_j^{t_j}(q_{j-1})) \right) \\
    &= d(h_n^{t_n} \circ \cdots \circ h_{j+1}^{t_{j+1}} ) \left( V_j(q_j) \right) \\
    &= d(h_n^{t_n} \circ \cdots \circ h_{j+1}^{t_{j+1}} ) \left( \beta(|q_j|) \mathbf{e}_j \right) \\
    &= \beta(|q_j|) \cdot d(h_n^{t_n} \circ \cdots \circ h_{j+1}^{t_{j+1}} ) \left( \mathbf{e}_j \right) \\
    &= \beta(|q_j|) \left( \mathbf{e}_j + \alpha_{j+1, j} \mathbf{e}_{j+1} + \cdots + \alpha_{n, j} \mathbf{e}_{n}  \right),
    \end{align*}
for some $\alpha_{i, j} \in \RR$. In particular, $\partial_t(ev_{\psi,p})_{(t,v)}$ is lower triangular, with diagonal entries $\beta(|q_1|), \dots, \beta(|q_n|)$. Since $|q_j| < 3$ and thus $\beta(|q_j|) > 0$ for all $1 \leq j \leq n$, it follows that $\partial_t(ev_{\psi,p})_{(t,v)}$ is surjective.

For the final statement, let $\sigma_p, \sigma_q \in \gr_k(T\BB^n) = \BB^n \times \gr_k(\RR^n)$, say $\sigma_p = (p, \sigma_1), \sigma_q = (q, \sigma_2)$. Pick $A \in SO(n, \RR)$ such that $A\sigma_1 = \sigma_2$ and $v \in B_{\mathfrak{so}_n}(0, R)$ such that $\exp(v) = A$. Since $L(v)|_{\BB^n} = A$ we have 
$$
d(L(v))(\sigma_p) = (Ap, A\sigma_1) = (Ap, \sigma_2).
$$
Observe that, for all $1 \leq j \leq n$, if $b$ and $h_j^{t_j}(b) \in \BB^n_2$, then $h_j^{t_j} : x \mapsto x + t_j \mathbf{e}_j$ in a neighborhood of $b$, so that $d(h_j^{t_j})_b = \id$. Writing $q = (q_1, \dots, q_n)$, $r = (r_1, \dots, r_n) = Ap$, and $t = (t_1, \dots, t_n) = q - r \in \BB^n_2$, for all $0 \leq j \leq n$  we have 
$$
h_j^{t_j} \circ \cdots \circ h_1^{t_1} (r) = (q_1, \dots, q_j, r_{j+1}, \dots, r_n).
$$
Thus $|h_j^{t_j} \circ \cdots \circ h_1^{t_1} (Ap)| \leq |q| + |r| < 2$. By the comment above, we inductively conclude that
$$
d(h_n^{t_n} \circ \cdots \circ h_1^{t_1})_{Ap} = \id.
$$
It follows that 
\begin{align*}
d(\psi(t,v))(\sigma_p) &= d(h_n^{t_n} \circ \cdots \circ h_1^{t_1}) \circ d(L(v))(\sigma_p) \\
&= d(h_n^{t_n} \circ \cdots \circ h_1^{t_1})(Ap, \sigma_2) \\
&= (q, \sigma_2) = \sigma_q, 
\end{align*} 
so we are done.
\end{proof}

 \subsection{Step 2: Patching together.} Choose a finite collection $\{ (U_i, \phi_i) \}_{i=1}^L$ of smooth charts $\phi_i : U_i \xrightarrow[]{ \ \ \cong \ \ } \BB^n_4$ for $M$ such that
 $$
 M = \ds\bigcup_{i=1}^L \phi_i^{-1}(\BB^n_1). 
 $$
 Set $W_i = \phi_i^{-1}(\BB^n_1)$, and let $D$ denote the diameter of the \v{C}ech $1$-complex associated to the cover $\{ W_i \}_{i=1}^L$ of $M$.\footnote{Explicitly, $D$ is the smallest number such that for all $1 \leq k, \ell \leq L$, there exist $i_0, i_1, \dots, i_{D-1}, i_D \in \{1, \dots, L\}$ such that $i_0 = k$, $i_D = \ell$, and $W_{i_{j}} \cap W_{i_{j+1}} \neq \varnothing$ for all $0 \leq j \leq D-1$.} Since $M$ is connected clearly we have $D \leq L-1 < \infty$.
 
 Now for each $i = 1, \dots, L$ we can define
 $$
 \psi_i : \RR^{n} \times \mathfrak{so}_n \longrightarrow \Diff^\infty(M)
 $$
 by 
\begin{equation*}
\psi_i(t, v) =
    \begin{cases}
        \phi_i^{-1} \circ \psi(t, v) \circ \phi_i & \text{on } U_i \\
        \text{id} & \text{on } M \setminus U_i. 
    \end{cases}
\end{equation*}
Then define
$$
\widehat{\psi} : \left( \RR^{n} \times \mathfrak{so}_n \right)^L \longrightarrow \Diff^\infty(M)
$$
by $\wh{\psi} = \operatorname{comp} \circ (\psi_1 \times \cdots \times \psi_L)$, i.e.
$$
\widehat{\psi}((t_1, v_1), \dots, (t_L, v_L)) = \psi_L(t_L, v_L) \circ \cdots \circ \psi_1(t_1, v_1).
$$
Finally, define 
$$
\Psi : ( ( \RR^{n} \times \mathfrak{so}_n )^L )^{D+1} \to \Diff^\infty(M)
$$
by 
$$
\Psi(q_1, \dots, q_{D+1}) = \widehat{\psi}(q_{D+1}) \circ \cdots \circ \widehat{\psi}(q_1).
$$

For convenience, let $N = \left( 2n + {n \choose 2} \right) (D+1) L$, and fix some isomorphism $( ( \RR^{2n} \times \mathfrak{so}_n )^L )^{D+1} \cong \RR^N$. Then we have

\begin{prop}
    The evaluation map
    $$
    ev_\Psi : \RR^N \times M \longrightarrow M \times M
    $$ 
    associated to $\Psi$ is a $C^\infty$ submersion. Moreover, if for all $0 \leq k \leq n$ we define 
    \begin{align*}
    \widehat{ev}^k_\Psi : \RR^n \times \gr_k(TM) &\longrightarrow \gr_k(TM) \times \gr_k(TM) \\
    (w, \sigma) &\longmapsto (\sigma, d(\Psi(w))(\sigma)), 
    \end{align*}
    then $\widehat{ev}^k_\Psi\big|_{\BB^N_R \times \gr_k(TM)}$ is surjective for all $R$ sufficiently large. 
\end{prop}

\begin{proof} The fact that $ev_\Psi$ is a smooth submersion follows immediately from Lemma \ref{lem:compsub} and Proposition \ref{prop:localcon}. The surjectivity of $\widehat{ev}^k_\Psi\big|_{\BB^N_R \times \gr_k(TM)}$ for $R \gg 0$ follows from the second half of Proposition \ref{prop:localcon} and the definition of $D$.
\end{proof}

\begin{cor}
    The hypotheses of Theorem \ref{thm:main} are satisfied by the family $(\overline{\BB}^N_R, \Psi)$.
\end{cor}

\section{Appendix A: Submersions from Manifolds with Boundary}

Our inclusion of the case $\bd \mc{H} \neq \varnothing$ necessitates a discussion of some basic properties of submersions from compact Riemannian manifolds with boundary. These maps are messier than submersions from compact manifolds without boundary, essentially because their fibers can intersect the boundary badly. Most dramatically, the fibers need not be manifolds (even with boundary) in general. For similar reasons, submersions from manifolds with boundary need not be open maps at boundary points, and (most relevant for our purposes) the locally trivial fibration property enjoyed by proper submersions between manifolds without boundary (known as Ehresmann's theorem) fails completely in this setting.\footnote{All these pathologies are easily illustrated by maps of the form $\pi : D \to \RR$ where $D \subset \RR^2$ is a compact planar domain with smooth boundary curves and $\pi$ denotes (the restriction of) projection onto some one-dimensional subspace. One can of course also obtain examples with closed codomain by post-composing $\pi$ with a smooth covering map $\RR \to \mathbb{S}^1$.}

Here we circumvent these issues rather naively, by restricting our maps to the interiors of the manifolds in question. An obvious drawback of this approach is the loss of compactness. Nevertheless, this trick turns out to be sufficient for our purposes because the geometry of such maps is still tamer than that of submersions from general open manifolds. The purpose of this section is to state and prove the needed constraints. 

Let $f : M^m \to N^n$ be a $C^1$ map between smooth manifolds without boundary. Recall that a \emph{box chart} (or \emph{submersion chart}) adapted to $f$ on $M$ is a triple $(B, \phi, \psi)$ consisting of an open set $B \subset M$ and a pair of $C^1$ diffeomorphisms $\phi : B \to (-r,r)^m$ and $\psi : f(B) \to (-r,r)^n$ such that $\psi \circ f \circ \phi^{-1} : (-r, r)^m \to (-r,r)^n$ is given by projection onto the first $n$ coordinates. For convenience we will use the following nominal refinement of this notion: 

\begin{defn}
We say that a box chart $(B, \phi_B, \psi_B)$ is \emph{nice} if $B$ is precompact, and there is another box chart $(B', \phi_{B'}, \psi_{B'})$ such that $\overline{B} \subset B'$, $\phi_B = \phi_{B'}|_B$, and $\psi_B = \psi_{B'}|_{f(B)}$. In this case we say that the box chart $B'$ \emph{cocoons} $B$. 
\end{defn}

The fundamental fact about submersions from manifolds without boundary is that their domains can be covered by box charts. Observe that nothing is lost in this statement by restricting our attention to nice box charts, since for $f : M \to N$ a smooth map with $\bd M = \varnothing$ it is immediate that
\begin{align*}
f \text{ is a submersion } &\Longleftrightarrow f \text{ can be covered by box charts } \\ 
&\Longleftrightarrow f \text{ can be covered by nice box charts.}
\end{align*}

\noindent On the other hand, we have the following straightforward fact:

\begin{lem} \label{lem:plaques}
Let $f : M^m \to N^n$ be a $C^1$ submersion between smooth Riemannian manifolds without boundary and $(B, \phi_B, \psi_B)$ a nice box chart adapted to $f$. Then there are constants $0 < C_L^B < C_U^B$ such that 
$$
C_L^B \leq \vol_{m-n}(f^{-1}(q) \cap B) \leq C_U^B
$$
for all $q \in f(B)$.
\end{lem}

\begin{proof}
Such an inequality holds trivially for the standard projection $\pi_n : \RR^m \to \RR^n$ and any (precompact) box $(-r, r)^m \subset \RR^m$, with $C_L = C_U = r^{m-n}$. For the general case, let $B$ be a nice box chart for $f$ and $(B', \phi_{B'}, \psi_{B'})$ a box chart cocooning $B$. By definition $\overline{B} \subset B'$ is compact and $\phi_B'$ is $C^1$; it follows that $\phi_{B'}\big|_{B} : B \hookrightarrow \RR^m$ distorts Riemannian metrics by a uniformly bounded amount, i.e. for some $C>1$ we have $C^{-1}g \leq (\phi_{B'})^*\eta \leq Cg$ on $B$, where $g$ denotes the metric on $M$ and $\eta$ the (Euclidean) metric on $\RR^m$. Moreover, for all $q \in f(B)$ we have
$$
f^{-1}(q) \cap B = \phi_{B'}^{-1} \big( \pi_n^{-1}(\psi_{B'}(q)) \cap \phi_{B'}(B) \big).
$$
The general result then follows trivially from the Euclidean case. \end{proof}

\begin{lem} \label{lem:appvolbounds}
Let $f : M^m \to N^n$ be a $C^1$ submersion between compact Riemannian manifolds, with $\bd M \neq \varnothing$ and $\bd N = \varnothing$. Then $\vol_{m-n}(f^{-1}(q) \cap \ucirc{M})$ is finite for all $q \in N$, and in fact uniformly bounded above. If $f\big|_{\ucirc{M}}$ is also surjective, then $\vol_{m-n}(f^{-1}(q) \cap \ucirc{M})$ is bounded away from zero. 
\end{lem}

\begin{proof} Let $\mathcal{F}_q := f^{-1}(q) \cap \ucirc{M} = ( f|_{\ucirc{M}} )^{-1}(q)$, and note that, since $\ucirc{M}$ is a manifold without boundary and $f\big|_{\ucirc{M}}$ is a submersion, each $\mathcal{F}_q$ is either empty or an embedded $(m-n)$-dimensional submanifold of $\ucirc{M}$. Now define
$$
\nu(q) := \vol_{m-n}(\mathcal{F}_q) = \ds\int_{\mathcal{F}_q} d\mathcal{F}_q.
$$
A priori we have $\nu : N \to [0, +\infty]$. We first show that in fact $\sup \nu < +\infty$. 

\textbf{Upper bound:} By attaching a collar $C_\epsilon \cong \bd M \times [0, \epsilon)$ to $M$ and extending the metric, we obtain an open Riemannian manifold without boundary $\widehat{M}$ containing $M$ as an isometrically embedded compact submanifold. Choose an extension of $f$ to a $C^1$ map $F : \widehat{M} \to N$; since the set $\{ p \in \widehat{M} \ : \ dF_p\text{ is surjective} \}$ is an open set containing $M$, by shrinking the collar width $\epsilon$ we may assume that $F$ is a submersion as well.

 Now pick a cover $\mc{B}$ of $\widehat{M}$ by nice box charts adapted to $F$; by compactness of $M$ we can find a finite subset $\{B_1, \dots, B_k\} \subset \mc{B}$ such that $M \subset B_1 \cup \cdots \cup B_k$. By Lemma \ref{lem:plaques} there are $C_U^i < +\infty$, $1 \leq i \leq k$, such that $\vol(F^{-1}(q) \cap B_i) < C_U^i$ for all $q \in N$. Since clearly every fiber $F^{-1}(q)$ intersects each $B_j$ in at most one plaque, for all $q$ we have 
\begin{align*}
\nu(q) = \vol(f^{-1}(q) \cap \ucirc{M}) & \leq \vol \left( F^{-1}(q) \cap \left( \ds\bigcup_{j=1}^k B_j \right) \right) \\
& \leq \ds\sum_{j=1}^k \vol \left(F^{-1}(q) \cap B_j \right) \\
& \leq k \cdot \ds\max_{1 \leq j \leq k} C_U^j < +\infty.
\end{align*}
\textbf{Lower bound:} It remains to show that $\inf \nu > 0$ if $\ucirc{f} := f|_{\ucirc{M}}$ is surjective. This assumption implies that for all $q \in N$, there is a nice box chart $E_q \subset \ucirc{M}$ adapted to $\ucirc{f}$ such that $E_q \cap \ucirc{f}^{-1}(q) \neq \varnothing$. For each such $E_q$, $\ucirc{f}(E_q) = f(E_q) \subset N$ is an open neighborhood of $q$, and
$$
N = \ds\bigcup_{q \in N} f(E_q). 
$$
Since $N$ is compact we may extract a finite subcover $\{f(E_1), \dots, f(E_\ell)\}$; in particular, every fiber $\mc{F}_q = \ucirc{f}^{-1}(q)$ intersects and thus contains a fiber plaque in some $E_j$, $1 \leq j \leq \ell$. On the other hand, the lower bound in Lemma \ref{lem:plaques} implies that every plaque in any $E_j$ has volume at least $C_L := \min \{ C_L^{E_1}, \dots, C_L^{E_\ell} \} > 0$, so 
$$
\nu(q) = \vol(\ucirc{f}^{-1}(q)) \geq C_L > 0
$$
for all $q \in N$, as desired. \end{proof}

\begin{lem} \label{lem:appfiberbounds} Let $f, M, N$ be as in Lemma \ref{lem:appvolbounds}. Suppose $\phi \in C^0(M, \RR_{\geq 0})$ is such that for all $q \in N$, we have $\phi(p) > 0$ for some $p \in \mathcal{F}_q = f^{-1}(q) \cap \ucirc{M}$. Then there exists $C_\phi > 1$ such that 
$$
\dfrac{1}{C_\phi} \leq \ds\int_{\mathcal{F}_q} \phi \ d\mathcal{F}_q \leq C_\phi
$$
for all $q \in N$.
\end{lem}

\begin{proof} Let $C_1 = \ds\sup_{p \in M} \phi < +\infty$. Then for all $q \in N$ we have
$$
\mathcal{I}_\phi(q) := \ds\int_{\mathcal{F}_q} \phi \ d\mathcal{F}_q \leq C_1 \vol(\mathcal{F}_q) \leq C C_1, 
$$
where $C$ is the upper bound on fiber volumes guaranteed by Lemma \ref{lem:appvolbounds}. 

Thus it remains only to show that $\mathcal{I}_\phi(q)$ is uniformly bounded away from $0$. This follows from an easy adaptation of the proof of the fiber volume lower bound. Consider the set $\mathcal{B}$ of all nice box charts $B_i$ for $f|_{\ucirc{M}}$ such that
$$
\ds\inf_{p \in B_i} \phi(p) > 0. 
$$
Since $\phi$ is continuous and positive somewhere on each fiber of $f|_{\ucirc{M}}$, the image of $\mathcal{B}$ under $f$ forms an open cover of $N$. Extract a finite subcover $\{ f(B_1), \dots, f(B_k) \} \subset f(\mathcal{B})$. Then for some $\delta > 0$ we have $\phi(p) > \delta$ for all $p \in B_1 \cup \dots \cup B_k$, and by Lemma \ref{lem:plaques} there exists $C_2 > 0$ such that the volumes of the plaques of each $B_i$ are uniformly bounded below by $C_2$. Since $\{ f(B_i) \}_{i=1}^k$ covers $N$, each $\mathcal{F}_q$ contains at least one full plaque in some $B_i$; thus
$$
\ds\int_{\mathcal{F}_q} \phi \ d\mathcal{F}_q \geq \delta C_2
$$
for all $q \in N$. Setting
$$
C_\phi = \max \left\{ CC_1, \dfrac{1}{\delta C_2} \right\}
$$
we are done. \end{proof}

\section{Appendix B: Normal Jacobians formula} \label{sec:njform}

Recall that our proof of Theorem 1.1 hinges on the existence of uniform bounds for the fiber integrals
$$
\ds\int_{\mathcal{V}_{p,q}} \dfrac{\nj(\pi_1)}{\nj(\pi_2)} \ d\mathcal{V}_{p, q}.
$$
The key point is that these integrals can be identified with fiber integrals of a \emph{universally defined} function $\eta$ on a slightly larger compact manifold with boundary which submerses onto $\gr_k(TM) \times \gr_{n-k}(TM)$. By the results of Appendix A, it then suffices to show that $\eta$ is (i) continuous, and (ii) positive somewhere on the interior of each fiber.

There are two main benefits to this approach. First, by using only what is required to apply Lemma \ref{lem:appfiberbounds}, it avoids obscuring the logic of the argument. Second, and largely as a consequence, it generalizes easily to the case when $\dim(V) + \dim(W) > M$.

The downside is that the constant $C$ remains somewhat opaque. In particular, this method gives no further analysis of the function $\eta$, on which $C$ depends, or equivalently of the function(s)
$$
\dfrac{\nj(\pi_1)}{\nj(\pi_2)} : \mc{V} \xrightarrow[ \ \ \ ]{} \RR_{\geq 0}.
$$
This is unsatisfying, especially given the fairly straightforward geometric intuition motivating the bounds. In addition, the mere existence of some constant is clearly insufficient for applications where more precise control over the constant is needed.

In the complementary dimensional case, a more transparent expression for $\frac{\nj(\pi_1)}{\nj(\pi_2)}$ is available, which in particular allows us to decompose $C$ into more tractable factors:

\begin{prop} \label{prop:njform} For $(\mathcal{H}, \psi)$ satisfying the hypotheses of Theorem \ref{thm:main} and any complementary dimensional submanifolds $V, W \subset M$, we have 
$$
\dfrac{\nj(\pi_1)}{\nj(\pi_2)}(h, p, q) = \dfrac{| \det(J_{(h, T_p V, T_q W)}) |}{\nj(d(ev_p)_h)} 
$$
for all $(h,p,q) \in \mathcal{V}$. Here $J$ is defined as in Claim \ref{claim:tangentspace}, and $ev_p$ as usual denotes the map
$$
ev_p := ev_{\psi}(\cdot,p) : \mathcal{H} \longrightarrow M
$$
given by $ev_p(h) = h(p)$.
\end{prop}
\noindent The purpose of this section is to prove this statement. We begin with

\begin{lem} \label{lem:bcsslem}
Let $V$, $W$ be finite dimensional inner product spaces and suppose $S \subset V \times W$ is a linear subspace such that $\pi_1|_S$ is an isomorphism. Then if 
$$
G := \pi_2 \circ (\pi_1|_S)^{-1} : V \to W,
$$
we have $S = \Gamma(G)$. If in addition $\pi_2|_S$ is surjective and $S$ is endowed with the inner product inherited from the product structure on $V \times W$, then
\begin{equation} 
\det(GG^*)^{-1/2} = \dfrac{|\det(\pi_1|_S)|}{\det((\pi_2|_S) \cdot (\pi_2|_S)^*)^{1/2}}.
\end{equation} 
\end{lem}

\begin{proof}
The first statement is obvious. The second follows from some elementary linear algebra computations; see \cite[pp. 241-243]{bcss} for details. 
\end{proof}

\begin{remark} In terms of normal Jacobians, this result specifies a context in which the innocent-looking formula
\begin{equation} \label{eqn:njinv}
\dfrac{1}{\nj(AB^{-1})} = \dfrac{\nj(B)}{\nj(A)}
\end{equation}
holds. Note however that this formula fails in general: normal Jacobians of general (not necessarily invertible) linear maps between inner product spaces are not typically multiplicative under composition. In particular, the fact that (\ref{eqn:njinv}) holds in the setting of Lemma \ref{lem:bcsslem} depends on the assumption that the metric on $V \times W$ (and thus the metric on $S$) is the product metric induced by the metrics on $V$ and $W$.
\end{remark}

\noindent \emph{Proof of Proposition \ref{prop:njform}.} As in the discussion following Claim \ref{claim:tangentspace}, observe that
$$
\nj(\pi_1)(h, p, q) = | \det((d\pi_1)_{(h, p, q)}) | = 0 \ \ \Longleftrightarrow \ \ \det(J_{(h, T_p V, T_q W)}) = 0.
$$
\noindent Thus it suffices to consider the case when both $d\pi_1$ and $J$ are invertible. Let $(h,p,q)$ be a regular point for $\pi_1$. Then we can define
$$
G \ := \ (d\pi_2)_{(h,p,q)} \circ (d\pi_1)_{(h,p,q)}^{-1}  :  T_h \mathcal{H}  \longrightarrow  T_p V \times T_q W.
$$
Since $\Gamma(G) = T_{(h,p,q)}\mathcal{V},$ with metric induced by the product metric on $T_h \mathcal{H} \times T_p V \times T_q W,$ Lemma \ref{lem:bcsslem} applies to give
\begin{equation} \label{eqn:G1}
\det(GG^*)^{-1/2} = \dfrac{\nj(\pi_1)}{\nj(\pi_2)}(h, p, q).
\end{equation}
Moreover, the description of $T_{(h,p,q)} \mathcal{V}$ given in Claim \ref{claim:tangentspace} immediately implies that the following diagram commutes:
\begin{center}
\begin{tikzcd}
& \arrow{ldd}[above]{d\pi_1 \ \ \ } T_{(h,p,q)}\mathcal{V} \arrow[rdd, "d\pi_2"] & \\
& & \\
T_h \mathcal{H} \arrow{rdd}[below]{d(ev_p)_h \ \ \ \ \ } \arrow{rr}[above]{G} & & T_p V \times T_q W \arrow{ldd}[below]{\ \ \ \ \ \ \ \ \ \ \ \ \ J_{(h, T_p V, T_q W)}} \\
& & \\
 & T_q M &
\end{tikzcd}
\end{center}
In particular, 
$$
G = J_{(h, T_p V, T_q W)}^{-1} \circ d(ev_p)_h, 
$$
so by (\ref{eqn:G1}) we compute 
\begin{align} \label{eqn:prop33proof}
\dfrac{\nj(\pi_1)}{\nj(\pi_2)}(h, p, q) &= \det(GG^*)^{-1/2} \nonumber \\
    &= \det((J^{-1} \circ d(ev_p)_h) \cdot (J^{-1} \circ d(ev_p)_h)^*)^{-1/2} \nonumber \\
    &= \det(J^{-1} \cdot d(ev_p)_h \cdot d(ev_p)_h^* \cdot (J^{-1})^*)^{-1/2}  \\
    &= \left[ \det(J^{-1}) \cdot \det(d(ev_p)_h \cdot d(ev_p)_h^*) \cdot \det((J^{-1})^*)  \right]^{-1/2} \nonumber \\
    &= \left[ \det(J^{-1})^2 \cdot \det(d(ev_p)_h) \cdot d(ev_p)_h^*) \right]^{-1/2} \nonumber \\
    &= \dfrac{|\det(J)|}{\det(d(ev_p)_h) \cdot d(ev_p)_h^*)^{1/2}} \nonumber.  \qed
\end{align}

\begin{remark}
   Up to a uniformly bounded factor, we have
   $$
   |\det \left( J_{(h, \sigma_p, \sigma_q)} \right) | \approx \sin \measuredangle(dh(\sigma_p), \sigma_q).
   $$
\end{remark}
 \noindent The multiplicative error involved is at most $\eta^k$, where $\eta$ is the supremum of $\max \left\{ \norm{dh_q}, \norm{dh_q^{-1}} \right\}$ over all $h \in \mc{H}$ and $q \in M$.
 
Similarly, the fact that $ev$ is a $C^1$ submersion implies that $\det(d(ev_p)_h) \cdot d(ev_p)_h^*)^{1/2}$ is uniformly bounded above and away from $0$, i.e. constant up to a uniform factor. 

Combining both errors we conclude that 
$$
\dfrac{\nj(\pi_1)}{\nj(\pi_2)}(h, p, q) \approx \sin \measuredangle(dh(\sigma_p), \sigma_q)
$$
up to a bounded multiplicative constant (which is itself a function of the constants bounding $\nj(ev_p)_h, \norm{dh_p}$, and $\norm{dh_p^{-1}}$ for all $(h, p) \in \mc{H} \times M$).  

This means that one can view Theorem \ref{thm:main} as reducing to the following claim: 

\begin{prop} \label{prop:anglebound}
    For all $0 \leq k \leq n$ there is a constant $C_k > 1$ such that for all $(\sigma_p, \sigma_q) \in \gr_k(TM) \times \gr_{n-k}(TM)$, 
    $$
    \dfrac{1}{C_k} \leq \ds\int_{h \in \mc{H}_{p,q}} \sin \measuredangle(dh(\sigma_p), \sigma_q) \ |d\mc{H}_{p,q}| 
    \leq C_k.
    $$
\end{prop}
    \noindent Here $\mc{H}_{p,q} = \{ h \in \ucirc{\mc{H}} : h(p) = q \}$; this is a submanifold of $\ucirc{\mc{H}}$, with Riemannian metric induced from $\mc{H}$, which can be canonically (and isometrically) identified with $\mc{V}_{p,q}$ (indeed $\mc{V}_{p,q} = \mc{H}_{p,q} \times \{ p \} \times \{ q \}$). 
    
  Using the results above, our original constant $C$ can be estimated as a function of the constants $C_k$ and the bounds on $\nj(ev_p)_h, \norm{dh_p}$, and $\norm{dh_p^{-1}}$.

\bibliographystyle{alpha}
\bibliography{main.bbl}

\end{document}